\documentclass[12pt,reqno]{amsart}

\usepackage{mathrsfs}

\usepackage{amsmath}
\usepackage{hyperref}
\usepackage{amssymb}
\usepackage{graphicx}
\usepackage{color}
\usepackage{latexsym}
\usepackage{comment}
\usepackage{stmaryrd}
\usepackage{MnSymbol}
\usepackage{amsthm}

\usepackage{url}
\usepackage{float}

\usepackage[mathlines]{lineno}

\newtheorem{theorem}{Theorem}

\newtheorem{proposition}[theorem]{Proposition}
\newtheorem{observation}[theorem]{Observation}

\theoremstyle{remark}
\newtheorem*{remark}{Remark}
\newtheorem*{remarks}{Remarks}

\numberwithin{equation}{section}

\def \sc {$\mathrm{SC}(k,r)$}
\def \sckr #1#2 {$\mathrm{SC}(#1,#2)$}
\def \tr {$\mathsf{tr}(k,r)$}
\def \trkr#1#2{\mathsf{tr}(#1,#2)}
\def \trNk#1#2{\mathsf{tr^*}(#1,#2)}
\def \tabc {$\Delta(a,b,c)$}

\def\coef#1{[#1]}
\def\Res{\operatorname{Res}}

\allowdisplaybreaks

\makeatletter
\let\start@align@nopar\start@align
\let\start@gather@nopar\start@gather
\let\start@multline@nopar\start@multline
\long\def\start@align{\par\start@align@nopar}
\long\def\start@gather{\par\start@gather@nopar}
\long\def\start@multline{\par\start@multline@nopar}
\makeatother

\begin{document}

\title[Counting triangulations of subdivided convex polygons]
{Counting triangulations of some classes \\ of subdivided convex polygons}

\author[A. Asinowski, C. Krattenthaler and T. Mansour]
{Andrei Asinowski$^*$, Christian Krattenthaler$^\dagger$ and Toufik Mansour$^\ddagger$}

\address{$^*$ Institut f\"ur Diskrete Mathematik und Geometrie,
Technische Universit\"{a}t Wien.
Wiedner Hauptstra\ss e 8--10, A-1040 Vienna, Austria.\newline
WWW: {\tt \url{http://dmg.tuwien.ac.at/asinowski/}}.}

\address{$^{\dagger}$ Fakult\"at f\"ur Mathematik, Universit\"at Wien.
Oskar-Morgenstern-Platz~1, A-1090 Vienna, Austria.
WWW: {\tt \url{http://www.mat.univie.ac.at/\~kratt/}}.}

\address{$^\ddagger$ Department of Mathematics, University of Haifa.
Abba Khoushy Ave 199, Mount Carmel, Haifa 3498838, Israel.
WWW: {\tt \url{http://math.haifa.ac.il/toufik/}}.}

\thanks{$^*$ Research supported by the Austrian
Science Foundation FWF, grant S50-N15,
in the framework of the Special Research Program
``Algorithmic and Enumerative Combinatorics".\newline\indent
$^\dagger$ Research partially supported by the Austrian
Science Foundation FWF, grant S50-N15,
in the framework of the Special Research Program
``Algorithmic and Enumerative Combinatorics".}


\begin{abstract}
We compute the number of triangulations of a convex $k$-gon
each of whose sides is subdivided by $r-1$ points.
We find explicit formulas and generating functions,
and we determine the asymptotic behaviour of these numbers
as $k$ and/or $r$ tend to infinity.
We connect
these results with the question of finding
the planar set of points in general position
that has the minimum possible number of triangulations ---
a well-known open problem from computational geometry.
\end{abstract}


\maketitle

\section{Introduction}

Let $k$ and $r$ be two natural numbers, $k\geq 3$, $r \geq 1$.
Let \sc\ denote a convex $k$-gon in the plane
each of whose sides is subdivided by $r-1$ points.
(Thus, the whole configuration consists of $kr$ points.)
In what follows, the exact measures are not essential:
without loss of generality, we may consider
a regular $k$-gon with sides
subdivided by evenly spaced points.
The $k$ vertices of the original (``basic'') $k$-gon
will be called \emph{corners}, and they will be denoted
(say, clockwise)
by $P_{0,0}, P_{1,0}, \dots, P_{k-1,0}$
(with arithmetic modulo~$k$ in the first index, so that
$P_{k,0}=P_{0,0}$).
The $r-1$ points that subdivide the segment $P_{i,0} P_{i+1,0}$
(oriented from $P_{i,0}$ to $P_{i+1,0}$)
will be denoted by
$P_{i,1}, P_{i,2}, \dots, P_{i,r-1}$
(we shall also occasionally write $P_{i,r}$ for $P_{i+1,0}$).
The subdivided segments $P_{i,0} P_{i+1,0}$ --- that is,
the point sequences of the form
$P_{i,0}, P_{i,1}, P_{i,2}, \dots, P_{i,r-1}, P_{i+1,0}$ ---
will be referred to as \emph{strings}.
Thus, the boundary of  \sc\ consists of $k$ strings,
and each corner belongs to two strings.
The reader is referred to Figure~\ref{fig:ex_hex} for an illustration.
For brevity, a convex polygon with subdivided edges
(not all of them necessarily subdivided by the same number of points) 
will be referred to as a
\emph{subdivided convex polygon}.
A subdivided convex polygon is \textit{balanced} if 
(as described above) all its sides are subdivided by the same number of points.

\begin{figure}
\begin{center}
\includegraphics[scale=0.85]{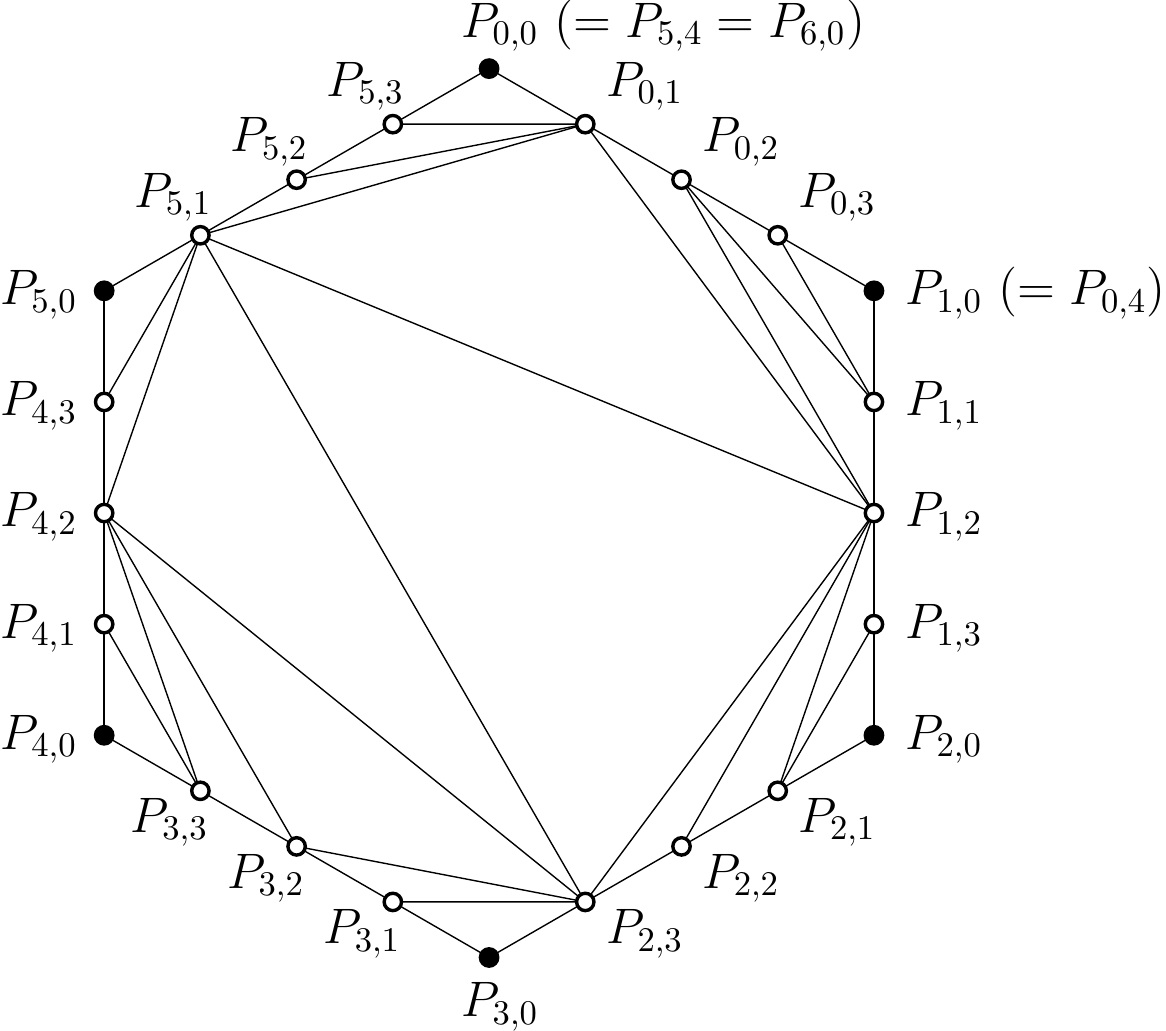}
\end{center}
\caption{The subdivided convex polygon \sckr{6}{4} \ and one of its
  triangulations.}
\label{fig:ex_hex}
\end{figure}

A \textit{triangulation} of a finite planar point set $\mathrm{S}$
is a dissection of its convex hull by non-crossing diagonals\footnote{
By a ``diagonal" we mean a straight-line segment connecting two points 
of the set $\mathrm{S}$.} into triangles.
We emphasize that maximal triangulations are meant;
in particular, no triangle can have another point of the set in the
interior of one of its sides.
The set of triangulations of a point set $\mathrm{S}$
will be denoted by $\mathsf{TR}(\mathrm{S})$.

Triangulations of (structures equivalent or related to)
subdivided convex polygons have appeared in earlier work.
Hurtado and Noy \cite{hn} considered triangulations of \emph{almost convex
  polygons}, which turn out to be equivalent to subdivided convex polygons
according to our terminology.
They dealt with the non-balanced case --- that is,
$k$-gons whose sides
are subdivided, but not necessarily into the same number of points.
In particular, Hurtado and Noy derived an
inclusion-exclusion formula for the number of triangulations of
a subdivided convex $k$-gon whose sides are subdivided by $a_1, a_2,
\dots, a_k$ points, and they showed that this number is independent
of the specific distribution of the subdivisions among the sides
of the basic $k$-gon.
On the other hand, Bacher and Mouton~\cite{b, bm} considered triangulations of more general
\emph{nearly convex polygons} defined as
infinitesimal perturbations of subdivided convex polygons.
They derived a formula for the number of triangulations of such polygons
in terms of certain polynomials that depend on the shape of chains.

The main purpose of the present paper is to present enumeration
formulas and precise asymptotic results for the number of triangulations
of a subdivided convex polygon in the balanced case, that is, where
each side of the polygon is subdivided into the same number of points.
Our enumeration formulas are more compact than those of Hurtado and Noy
or of Bacher and Mouton when specialised to the balanced case.
We shall as well provide formulas for some non-balanced cases.

Let us denote the number of triangulations of \sc \ by \tr.
For $r=1$ our configuration is just a convex $k$-gon,
and, thus,
$\trkr {k} {1} =C_{k-2}$,
where $C_n=\frac{1}{n+1}\binom{2n}{n}$
is the $n$th Catalan number.
It is easy to find $\trkr{k}{r}$ for small values of $k$ and $r$ by inspection.
For example, we have
$\trkr{3}{2}=4$,
$\trkr{3}{3}=29$
and
$\trkr{4}{2}=30$;
see Figure~\ref{fig:b24} (there, symmetries must also be taken into account;
for each triangulation it is shown how many different triangulations can be
obtained from it under symmetries).
\begin{figure}
\begin{center}
\includegraphics[scale=0.8]{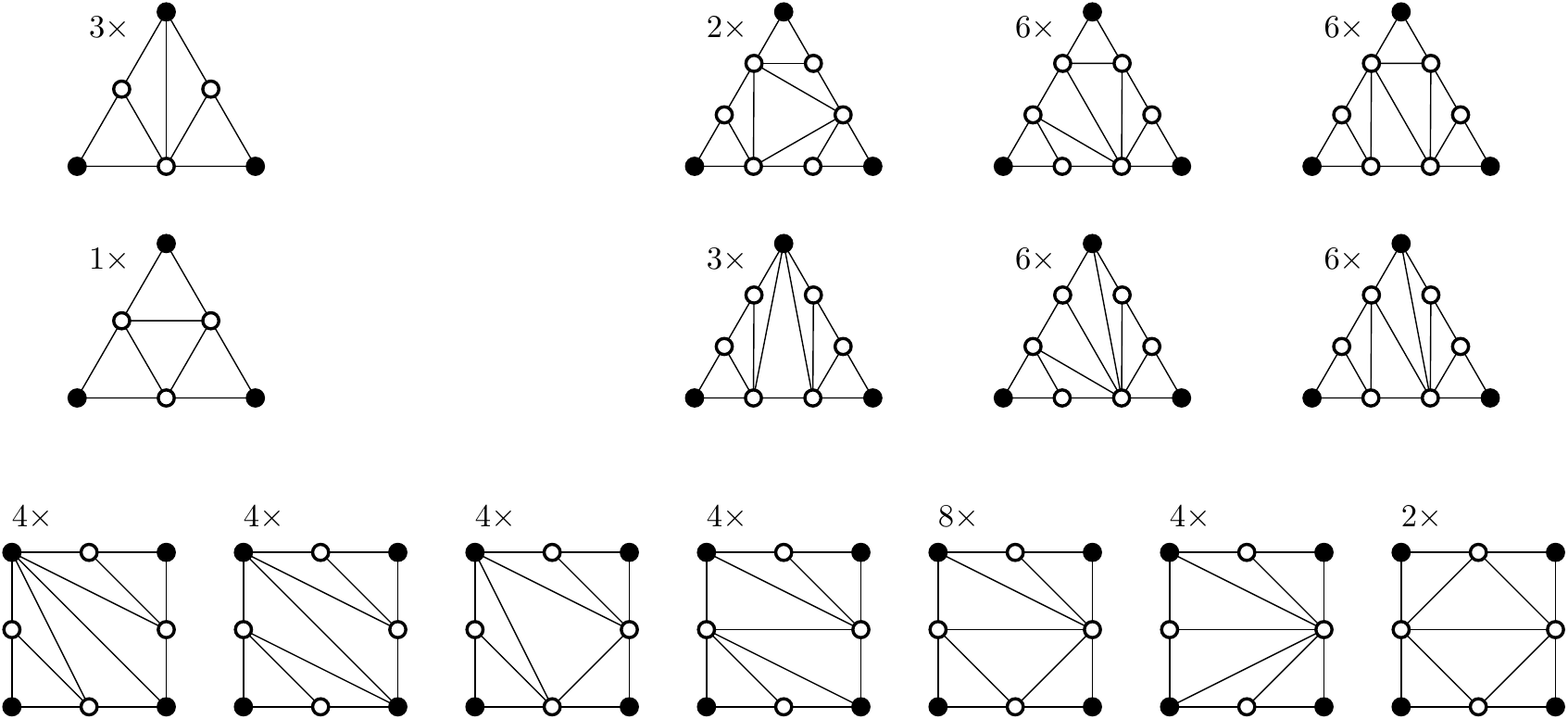}
\end{center}
\caption{All triangulations of \sckr{3}{2} ,
  \sckr{3}{3} \ and \sckr{4}{2} .}
\label{fig:b24}
\end{figure}
Values of $\trkr{k}{r}$
for $1\leq k \leq 7$, $1\leq r \leq 6$
are shown in Table~\ref{tab:values};
the meaning of these values for
$k=2$ --- the central binomial coefficients ---
will be explained in Section~\ref{sec:formula}
(see the remark after the proof of Theorem~\ref{thm:Formel}).
The sequence
$(\trkr{k}{2})_{k\geq 3}$ is OEIS/A086452, while
the sequence
$(\trkr{3}{r})_{r\geq 1}$ is OEIS/A087809~\cite{oeis}.

In the next section, we derive our formulas for the numbers $\trkr{k}{r}$.
They are given in the form of double sums, see Theorem~\ref{thm:Formel},
thus answering an open question posed in~\cite{hn}.
These formulas come from a representation of $\trkr{k}{r}$ in terms of
a complex contour integral (see Proposition~\ref{prop:Formel}), when
interpreted as a coefficient extraction formula. We use this integral
representation to prove in Section~\ref{sec:GF} that the ``vertical"
generating functions $\sum_{k\ge2}\trkr{k}{r}x^k$ as well as
the ``horizontal" generating functions
$\sum_{r\ge1}\trkr{k}{r}x^r$ are all algebraic. More precisely,
we find
explicit expressions for these generating functions in terms of roots
of certain (explicit) polynomials.
We devote a separate section, Section~\ref{sec:k=3}, to the special case
$k=3$, since in that case several alternative formulas that are
more attractive than the formulas in Theorem~\ref{thm:Formel} are available.
Moreover, in Section~\ref{sec:abc} we also
consider the {\it non-balanced\/} case of $k=3$:
we count triangulations of a triangle
whose sides are subdivided by $a$, $b$, and $c$ points, respectively.
The resulting compact formulas are presented in
Propositions~\ref{prop:abc_all} and~\ref{prop:abc_bl}.
Then, in Section~\ref{sec:asy},
we determine the asymptotic behaviour of $\trkr{k}{r}$ as $r$ and/or
$k$ tend to
infinity, see Theorems~\ref{thm:rinf} and~\ref{thm:kinf}.
This is achieved by transforming the contour integral into a complex
integral
along a line in the complex plane parallel to the imaginary axis that
passes through the saddle point of the integrand.
In the final Section~\ref{sec:dc},
we connect our results with
a well-known open problem from computational geometry:
the problem of
determining a planar set of $n$ points in general position
with the minimum number of triangulations.
We show that our results support
a conjecture of Aichholzer, Hurtado and Noy \cite{ahn} that this minimum is
attained by the so-called \textit{double circle}.

\begin{table}\label{tab:values}
\begin{center}
\scalebox{0.91}{
\begin{tabular}{|r||c|c|c|c|c|c|c|c|}
\hline
 & $r=1$ & $2$ & $3$ & $4$ & $5$ & $6$ \\ \hline\hline  
 $k=2$ & $1$ & $1$ & $2$ & $6$ & $20$ & $70$ \\ \hline  
$3$ & $1$ & $4$ & $29$ & $229$ & $1847$ & $14974$ \\ \hline  
$4$ & $2$ & $30$ & $604$ & $12168$ & $238848$ & $4569624$ \\ \hline  
$5$ & $5 $ & $250 $ & $13740 $ & $699310 $ & $33138675 $ & $1484701075
$ \\ \hline  
$6$ & $14 $ & $2236 $ & $332842 $ & $42660740 $ & $4872907670 $ &
$510909185422 $ \\ \hline  
$7$ & $42 $ & $20979 $ & $8419334 $ & $2711857491 $ & $745727424435 $
& $182814912101920 $ \\ \hline 
\end{tabular}
}
\end{center}
\caption{Values of $\trkr{k}{r}$ for $1\leq k \leq 7$, $1\leq r \leq 6$.}
\end{table}

\section{A formula for $\trkr{k}{r}$}\label{sec:formula}

In this section we derive
two --- very similar --- double sum formulas
for $\trkr{k}{r}$, given in \eqref{eq:Formel1} and \eqref{eq:Formel2}.
Starting point for finding these double sum expressions
is the inclusion-exclusion
formula \eqref{eq:b2}, which is equivalent to that found
in~\cite{hn} and in~\cite{b, bm}.
We include its derivation for the sake of completeness.

We start by ``inflating'' \sc.
That is, we replace its strings
by slightly curved circular arcs so that
a set of $kr$ points in convex position is obtained.
We keep the labels for these points.
Denote this point set by $\mathrm{C}(k  \cdot  r)$.
It is easy to see that each triangulation of \sc\
is transformed into
a triangulation of $\mathrm{C}(k  \cdot  r)$,
see Figure~\ref{fig:ex_hex_inj}.
\begin{figure}
\begin{center}
\includegraphics[scale=0.8]{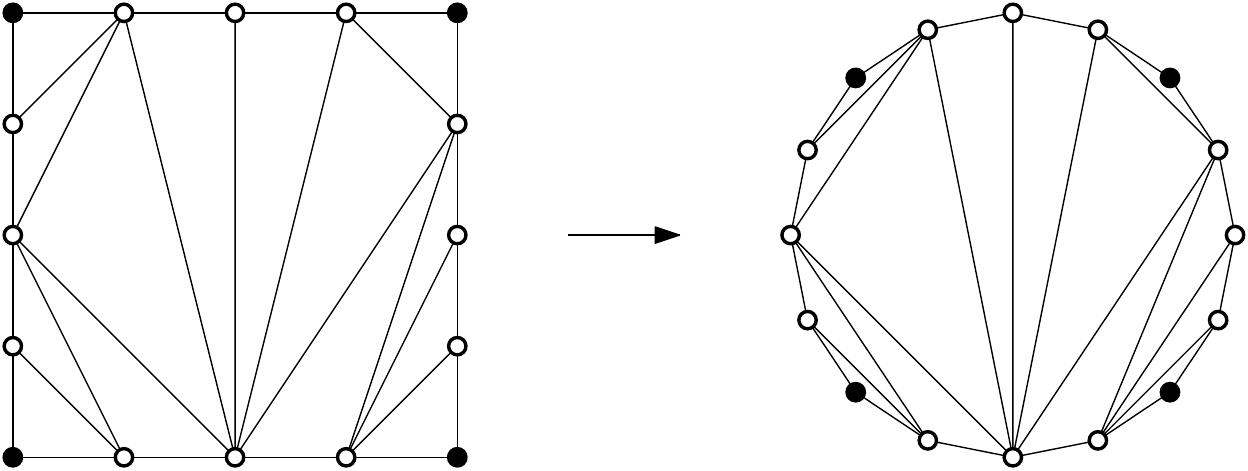}
\end{center}
\caption{Injection
$\varphi_{k,r}$ from
$\mathsf{TR}(\mathrm{SC}(k,r))$ to
$\mathsf{TR}(\mathrm{C}(k  \cdot  r))$}
\label{fig:ex_hex_inj}
\end{figure}
More formally, this ``inflation" defines a natural injection
$\varphi=\varphi_{k,r}$ from
$\mathsf{TR}(\mathrm{SC}(k,r))$ to
$\mathsf{TR}(\mathrm{C}(k  \cdot  r))$:
for each $D \in \mathsf{TR}(\mathrm{SC}(k,r))$,
triangulation $\varphi(D) \in \mathsf{TR}(\mathrm{C}(k \cdot r))$
uses the diagonals with the same labels as $D$.
Thus $\trkr{k}{r}$ is the size of the image of $\varphi$.
We say that a triangulation of
$\mathrm{C}(k \cdot r)$ is \textit{legal} if it belongs to the image
of $\varphi$ ---
that is, corresponds to a (unique) triangulation of \sc.
It is easy to see the following.

\begin{observation}\label{obs:legal_tr}
Let $T$ be a triangulation of\/ $\mathrm{C}(k \cdot r)$.
$T$ is legal if and only if it uses
no diagonal whose endpoints belong to the same string
(that is, to the set $\{P_{i,0}, P_{i,1}, \dots,\break P_{i,r-1},
P_{i+1,0}\}$ for some $i$).
\end{observation}

We call the diagonals mentioned in Observation~\ref{obs:legal_tr}
\textit{forbidden},
and we need to exclude triangulations that contain them
from the set of all the triangulations of $\mathrm{C}(k \cdot r)$.
Notice, however, that, if a triangulation of $\mathrm{C}(k \cdot r)$
uses some forbidden diagonal, then it necessarily (also) uses
a forbidden diagonal that connects two points
at distance $2$ along the boundary of $\mathrm{C}(k \cdot r)$.
Therefore, the characterization of legal triangulations
from Observation~\ref{obs:legal_tr} can be simplified as follows.

\begin{observation}\label{obs:legal_tr_2}
Let $T$ be a triangulation of $\mathrm{C}(k \cdot r)$.
$T$ is legal if and only if it uses
no diagonal of the form $P_{i,j}P_{i, j+2}$
with $0\leq i \leq k-1$ and $0 \leq j \leq r-2$.
\end{observation}

We call the diagonals mentioned in Observation~\ref{obs:legal_tr}
\textit{essentially forbidden}.
Figure~\ref{fig:ex_forb_diag} shows (a) forbidden
and (b) essentially forbidden diagonals of $\mathrm{C}(4 \cdot 4)$.
\begin{figure}
\begin{center}
\includegraphics[scale=0.8]{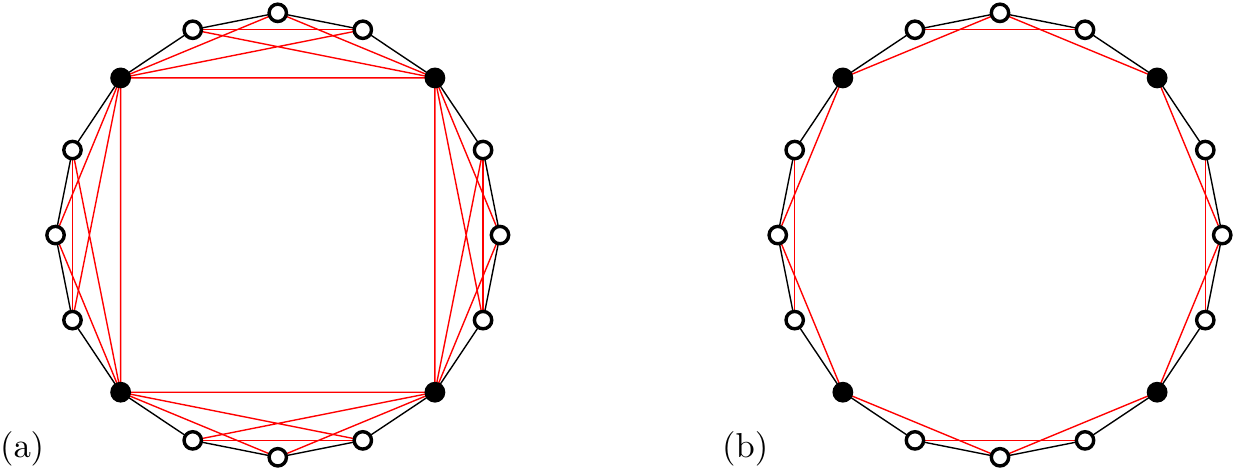}
\end{center}
\caption{Forbidden (a) and essentially forbidden (b) diagonals of
  $\mathrm{C}(4 \cdot 4)$.}
\label{fig:ex_forb_diag}
\end{figure}

Thus, we need to exclude triangulations of $\mathrm{C}(k \cdot r)$
that use essentially forbidden diagonals.
The total number of essentially forbidden diagonals is $k(r-1)$,
but the neighbouring essentially forbidden diagonals
(that is, $P_{i,j}P_{i, j+2}$ and $P_{i,j+1}P_{i, j+3}$
for some $i$ and $j$ with $0\leq i \leq k-1$ and $0 \leq j \leq r-3$)
cannot coexist in the same triangulation of $\mathrm{C}(k \cdot r)$.
Thus, the number of possible choices of $\ell$
essentially forbidden diagonals from the same string, where
$0 \leq \ell \leq
\lfloor r/2 \rfloor$,
equals the number of $\ell$-subsets of $\{1, 2, \dots, r-1\}$
that do not contain adjacent numbers.
This is a simple exercise in elementary combinatorics, and the
answer is $\binom{r-\ell}{\ell}$.
Therefore, the number of ways to choose $m$ pairwise non-crossing
essentially forbidden diagonals in $\mathrm{C}(k \cdot r)$ is
\[
{a}_{k,r,m}:=[x^m] \left(\sum_{\ell = 0}^{\lfloor r/2 \rfloor}
\binom{r-\ell}{\ell} x^{\ell} \right)^k,
\]
where $[x^m]f(x)$ denotes the coefficient of $x^m$ in the
polynomial of formal power series $f(x)$.

Once $m$ essentially forbidden diagonals of $\mathrm{C}(k \cdot r)$
are chosen,
we are left with a convex $(kr-m)$-gon to be triangulated.
Therefore, the number of illegal triangulations that use at least
$m$ essentially forbidden diagonals
is ${a}_{k,r,m} C_{kr-m-2}$.
At this point we can apply the inclusion-exclusion principle and obtain
\begin{equation}\label{eq:b1}
\trkr{k}{r} = \sum_{m=0}^{\lfloor r/2 \rfloor k} (-1)^m \, {a}_{k,r,m} \,
C_{kr-m-2}.
\end{equation}

Next, we observe that
\[
\sum_{\ell = 0}^{\lfloor r/2 \rfloor} \binom{r-\ell}{\ell} (-x)^{\ell}
= x^{r/2} \, U_r\left( \frac{1}{2 \sqrt{x}}\right),
\]
where $U_r(x)$ is the $r$th Chebyshev polynomial of the second kind.
Thus,
\[
(-1)^m \, {a}_{k,r,m} =
[x^m] \left( x^{r/2} \, U_r\left( \frac{1}{2 \sqrt{x}}\right) \right)^k,
\]
and \eqref{eq:b1} can be rewritten as
\begin{equation}\label{eq:b2}
\trkr{k}{r} =\coef{x^{rk-2}}
\left(\left(x^{r/2}U_r\left(\frac {1}
{2\sqrt x}\right)\right)^k \,
C(x)\right),
\end{equation}
where
\[C(x)=\frac {1-\sqrt{1-4x}} {2x}\]
is the generating function for Catalan numbers.
Since an explicit form of $U_r(x)$ is
\[
U_r(x) =
\frac{\left(x+\sqrt{x^2-1}\right)^{r+1}
-\left(x-\sqrt{x^2-1}\right)^{r+1}}{2\sqrt{x^2-1}},
\]
it follows that
\begin{multline*}
\trkr{k}{r}=\coef{x^{rk-2}}
\left(
\frac {1} {2^{(r+1)k}({1}- {4x})^{k/2}}\right.\\
\left.
\cdot
\left(
\left({1} +\sqrt{{1}- {4x}}\right)^{r+1}
-\left({1} -\sqrt{{1}- {4x}}\right)^{r+1}
\right)^k \,
\frac {1-\sqrt{1-4x}} {2x}
\right).
\end{multline*}

Using Cauchy's integral formula, we may write this expression
in terms of a complex contour
integral, namely as
\begin{multline}\label{eq:tr1}
\trkr{k}{r}=
\frac {1} {2\pi i}\int_{\mathcal C}
\frac {dx} {2^{(r+1)k+1}x^{rk}({1}- {4x})^{k/2}}\\
\cdot
\left(
\left({1} +\sqrt{{1}- {4x}}\right)^{r+1}
-\left({1} -\sqrt{{1}- {4x}}\right)^{r+1}
\right)^k
\left({1-\sqrt{1-4x}}\right),
\end{multline}
where $\mathcal C$ is a small contour encircling the origin
once in positive direction. Next we perform the substitution
$x= t(1-t)$, in which case $dx=(1-2t)\,dt$.
This leads us to the following integral representation of our
numbers $\trkr{k}{r}$.

\begin{proposition} \label{prop:Formel}
For all positive integers $k$ and $r$ with $rk\ge3$, we have
\begin{equation} \label{eq:intFormel}
\trkr{k}{r}=
-\frac {1} {4\pi i}\int_{\mathcal C}
\frac {dt} {t^{rk}(1-t)^{rk}({1}- {2t})^{k-2}}
\left(
\left({1} -t\right)^{r+1}
-t^{r+1}
\right)^k,
\end{equation}
where $\mathcal C$ is a contour close to $0$ which encircles
$0$ once in positive direction.
\end{proposition}

\begin{proof}
Carrying out the above described substitution in \eqref{eq:tr1},
we arrive at
\begin{equation} \label{eq:tr2}
\trkr{k}{r}=
\frac {1} {2\pi i}\int_{\mathcal C'}
\frac {(1-2t)\,dt} {t^{rk-1}(1-t)^{rk}({1}- {2t})^{k}}
\left(
\left({1} -t\right)^{r+1}
-t^{r+1}
\right)^k,
\end{equation}
where $\mathcal C'$ is a(nother) contour close to the origin encircling
the origin once in positive direction.
In order to obtain the more symmetric form (with respect to the
substitution $t\to1-t$) in \eqref{eq:intFormel}, we
blow up the contour $\mathcal C'$  so that
it is sent to infinity. While doing this, we must pass over
the pole $t=1$ of the integrand. (The point $t=1/2$ is a removable
singularity of the integrand.)
This must be
compensated by taking the residue at $t=1$ into account.
The integrand is of the order $O(t^{-rk+2})$ as $\vert
t\vert\to\infty$, and even of the order $O(t^{-rk+1})$ if $r$ is odd.
Together, this means that the integrand is of the order $O(t^{-2})$
as $\vert t\vert\to\infty$ for $rk\ge3$. Hence,
the integral along the contour near infinity vanishes. Thus, we obtain
\begin{align} \notag
\trkr{k}{r}&=-\Res_{t=1}
\frac {1} {t^{rk-1}(1-t)^{rk}({1}- {2t})^{k-1}}
\left(
\left({1} -t\right)^{r+1}
-t^{r+1}
\right)^k\\
&=
-\frac {1} {2\pi i}\int_{\mathcal C}
\frac {dt} {(1+t)^{rk-1}(-t)^{rk}(-{1}- {2t})^{k-1}}
\left(
\left(-t\right)^{r+1}
-(1+t)^{r+1}
\right)^k,
\label{eq:tr3}
\end{align}
where $\mathcal C$ is a contour close to $0$, which encircles
$0$ once in positive direction.
We have thus obtained two (slightly) different expressions for
$\trkr{k}{r}$, namely \eqref{eq:tr2} and \eqref{eq:tr3}.
Thus, $\trkr{k}{r}$ is also equal to their arithmetic mean.
If this is worked out, after having substituted $-t$ for $t$ in
\eqref{eq:tr3}, one arrives at \eqref{eq:intFormel}.
\end{proof}

We are now in the position to derive explicit formulas for
$\trkr{k}{r}$ in terms of binomial double sums.

\begin{theorem} \label{thm:Formel}
For all positive integers $k$ and $r$ with $rk\ge3$, we have
\begin{align} \label{eq:Formel1}
\trkr{k}{r}&=
\sum_{j=0}^k\sum_{\ell=0}^{r k-(r+1)j-2} (-1)^j\, 2^\ell\,
\binom kj \binom {k-2+\ell}{\ell}
\binom {(r-1)k-\ell-3}{r k-(r+1)j-\ell-2}\\
\label{eq:Formel2}
&=
\sum_{j=0}^k
\sum_{\ell=0} ^{rk-(r+1)j-1} (-1)^{j+1} \,2^{\ell-1}\,
\binom kj \binom {k-3+\ell}{\ell}
 \binom {(r-1)k-\ell-2} {rk-(r+1)j-\ell-1}.
\end{align}
\end{theorem}

\begin{proof}By Cauchy's integral formula,
Equation~\eqref{eq:tr2} can also be read as
\[
\trkr{k}{r}=[t^{rk-2}]
\frac {1} {(1-t)^{rk}({1}- {2t})^{k-1}}
\left(
\left({1} -t\right)^{r+1}
-t^{r+1}
\right)^k.
\]
If we now expand $\left(
\left({1} -t\right)^{r+1}
-t^{r+1}
\right)^k$ using the binomial theorem, and subsequently do the same
for powers of $1-t$ and of $1-2t$, then we are led to \eqref{eq:Formel1}.

If the same is done starting from \eqref{eq:intFormel}, then the
formula in \eqref{eq:Formel2} is obtained.
\end{proof}

\begin{remark}
If we choose $k=2$ in \eqref{eq:Formel2}, then
the only term which does not vanish is the one
with $j=1$ and $\ell=0$.
This term is $\binom {2r-4}{r-2}$, a central binomial coefficient.
If we interpret $\trkr{2}{r}$ (consistently with the case $k \geq 3$)
as the number of triangulations of $\mathrm{C}(2 \cdot r)$
that do not use (essentially) forbidden diagonals,
then it is easy to prove that this number
is indeed $\binom {2r-4}{r-2}$.
Indeed, one can construct a bijection
between such triangulations and
balanced sequences over $\{a, b\}$
using the same idea as in the proof of Theorem~\ref{prop:abc_all}(1) below.
 See Figure~\ref{fig:cbc} which illustrates this bijection for $r=4$.
\begin{figure}[h]
\begin{center}
\includegraphics[scale=0.8]{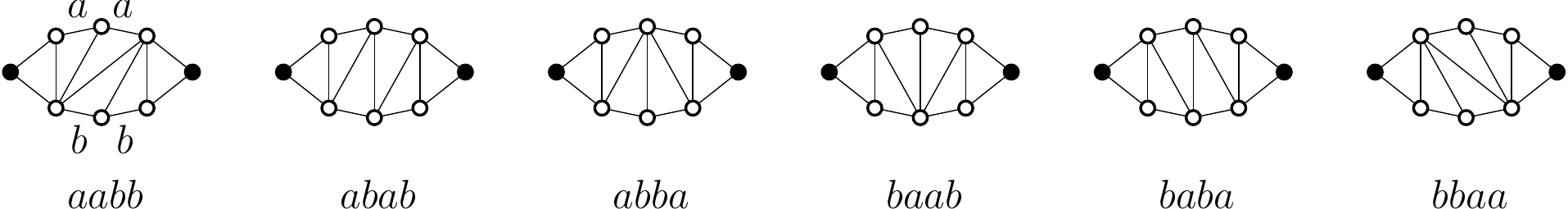}
\end{center}
\caption{Illustration of the fact $\trkr{2}{r}=\binom {2r-4}{r-2}$.}
\label{fig:cbc}
\end{figure}

\end{remark}

\section{Generating functions}\label{sec:GF}

Starting from the integral representation \eqref{eq:intFormel}, we
now show that ``horizontal" and ``vertical" generating functions
for the numbers $\trkr{k}{r}$ are algebraic.

\begin{theorem} \label{thm:vert}
For fixed $r\ge2$, we have
\begin{equation} \label{eq:vert}
\sum_{k\ge1}\trkr{k}{r}x^k=
-\frac {1} {2}\sum_{i=1}^r
\frac {t_i(x)^{r}(1-t_i(x))^{r}(1-2t_i(x))^2}
{(\frac {d} {dt}P_r)(x;t_i(x)) },
\end{equation}
where the $t_i(x)$, $i=1,2,\dots,r$, are the ``small" zeroes of the
polynomial\footnote{
$P_r(x;t)$ is indeed a polynomial in $t$
since $1-2t$ is a polynomial divisor of $(1-t)^{r+1}-t^{r+1}$.}
\[
P_r(x;t)=t^{r}(1-t)^{r}-x\left(\left({1} -t\right)^{r+1}
-t^{r+1}
\right)(1-2t)^{-1},
\]
that is, those zeroes $t(x)$ for which $\lim_{x\to0}t(x)=0$.
\end{theorem}

\begin{proof}
It should be noted that the right-hand side of \eqref{eq:intFormel}
vanishes for $k=0$. Hence,
multiplication of
both sides of \eqref{eq:intFormel} by $x^k$ and subsequent summation of both
sides over $k=0,1,\dots$ by means of the summation formula for
geometric series yield
\begin{align}
\notag
\sum_{k\ge1}\trkr{k}{r}x^k
&=
-\frac {1} {4\pi i}\int_{\mathcal C}
\frac {(1-2t)^2\,dt} {1-x\left(\left({1} -t\right)^{r+1}
-t^{r+1}
\right)t^{-r}(1-t)^{-r}(1-2t)^{-1} }\\
&=
-\frac {1} {4\pi i}\int_{\mathcal C}
\frac {t^{r}(1-t)^{r}(1-2t)^2}
{t^{r}(1-t)^{r}-x\left(\left({1} -t\right)^{r+1}
-t^{r+1}
\right)(1-2t)^{-1} }\,dt,
\label{eq:Pxt}
\end{align}
provided
\[
\vert x\vert<\left\vert
\frac {t^{r}(1-t)^{r}(1-2t)}
{\left({1} -t\right)^{r+1}
-t^{r+1}}
\right\vert
\]
for all $t$ along the contour $\mathcal C$.
By the residue theorem, this integral equals the sum of the residues
at poles of the integrand inside $\mathcal C$.
The poles are the ``small" zeroes of
the denominator polynomial
$P_r(x;t)$. By general theory, the zeroes $t_i(x)$ of $P_r(x;t)$,
$i=1,2,\dots,2r$, can be written in
terms of Puiseux series in $x$.
In order to identify the ``small" zeroes,
we write the equation $P_r(x;t)=0$ in the form
\[
\frac {t^{r}(1-t)^{r}(1-2t)}
{\left({1} -t\right)^{r+1}
-t^{r+1}}=x.
\]
Taking the $r$th root, we obtain
\[
\frac {t(1-t)(1-2t)^{1/r}}
{\left(\left({1} -t\right)^{r+1}
-t^{r+1}\right)^{1/r}}=\omega_r^i x^{1/r},
\quad \quad i=1,2,\dots,r,
\]
where $\omega_r=e^{2i\pi/r}$ is a primitive $r$th root of unity.
It is easy to see that there exists a unique power series
solution $t(X)$ to the equation
\[
\frac {t(1-t)(1-2t)^{1/r}}
{\left(\left({1} -t\right)^{r+1}
-t^{r+1}\right)^{1/r}}=X.
\]
We thus obtain the ``small" zeroes of $P_r(x;t)$ as
$t_i(x)=t(\omega_r^i x^{1/r})$, $i=1,2,\dots,r$. Because of the relation
$P_r(x;1-t)=P_r(x;t)$, the other zeroes of $P_r(x;t)$ are $1-t_i(x)$,
$i=1,2,\dots,r$, which are not ``small". The $t_i(x)$ for
$i=1,2,\dots,r$ are hence all ``small" zeroes.

In view of the above considerations, from \eqref{eq:Pxt} we get
\begin{align*}
\notag
\sum_{k\ge1}\trkr{k}{r}x^k
&=
-\frac {1} {4\pi i}\int_{\mathcal C}
\frac {t^{r}(1-t)^{r}(1-2t)^2}
{P_r(x;t) }\,dt\\
&=
-\frac {1} {2}\sum_{i=1}^r
\Res_{t=t_i(x)}
\frac {t^{r}(1-t)^{r}(1-2t)^2}
{P_r(x;t) }\\
&=
-\frac {1} {2}\sum_{i=1}^r
\frac {t_i(x)^{r}(1-t_i(x))^{r}(1-2t_i(x))^2}
{(\frac {d} {dt}P_r)(x;t_i(x)) },
\end{align*}
as desired.
\end{proof}

We illustrate this theorem by considering the case where $r=2$.
In this case, the polynomial $P_r(x;t)$ becomes
\[
P_2(x;t)=t^{2}(1-t)^{2}-x(t^2-t+1).
\]
The zeroes of this polynomial are
\[
t_i(x)=
\frac{1}{2}
   \left(1\pm\sqrt{1+2 x\pm2 \sqrt{x+4}
   \sqrt{x}}\right),
\quad \quad i=1,2,3,4.
\]
The small zeroes are
\[
t_1(x)=\frac{1}{2}
   \left(1-\sqrt{1+2 x-2 \sqrt{x+4}
   \sqrt{x}}\right)
\quad \text{and}\quad
t_2(x)=\frac{1}{2}
   \left(1-\sqrt{1+2 x+2 \sqrt{x+4}
   \sqrt{x}}\right).
\]
If all this is used in \eqref{eq:vert}, then we obtain
\begin{multline*}
\sum_{k\ge1}\trkr{k}{2}x^k=
\frac {1} {8}\sqrt{\frac x{x + 4}}
\left(
    \sqrt{1 + 2 x + 2 \sqrt{x (x + 4)}}
    \left(\sqrt{x} + \sqrt{x+4}\right)^2\right.\\
\left.
-\sqrt{1 + 2 x - 2 \sqrt{x (x + 4)}} \left(\sqrt{x} - \sqrt{x+4}\right)^2
\right)
\end{multline*}
after some simplification.

\begin{theorem} \label{thm:hor}
For fixed $k\ge2$, we have
\begin{equation} \label{eq:hor}
\sum_{r\ge1}\trkr{k}{r}x^r=
\frac {1} {2}
\sum_{j=0}^k (-1)^j\binom kj
\sum_{i=1}^{k-j}
\frac {t_{i,j}^{j+1}(x)(1-t_{i,j}(x))^{k-j+1}}
{({1}- {2t_{i,j}(x)})^{k-2}\,(k-j-kt_{i,j}(x))},
\end{equation}
where the $t_{i,j}(x)$, $i=1,2,\dots,k-j$, are the ``small" zeroes of the
polynomial
\[
Q_{j,k}(x;t)=t^{k-j}(1-t)^{j}-x,
\]
$j=1,2,\dots,k$,
that is, those zeroes $t(x)$ for which $\lim_{x\to0}t(x)=0$.
\end{theorem}

\begin{proof}
We multiply both sides of \eqref{eq:intFormel} by $x^r$ and then sum both
sides over $r=0,1,\dots$. Subsequently, we use the binomial theorem to
expand $\left((1-t)^{r+1}-t^{r+1}\right)^k$ and evaluate the resulting
sums over~$r$ by means of the summation formula for
geometric series.
Taking into account that the right-hand side of \eqref{eq:intFormel}
vanishes also for $r=0$, this leads us to
\begin{align}
\sum_{r=1}^\infty \trkr{k}{r}x^r&=
-\frac {1} {4\pi i}\int_{\mathcal C}
\frac {dt} {({1}- {2t})^{k-2}}
\sum_{j=0}^k (-1)^j\binom kj
t^{j}(1-t)^{k-j}
\frac {1} {1-xt^{-(k-j)}(1-t)^{-j}}
\notag\\
&=
-\frac {1} {4\pi i}\int_{\mathcal C}
\frac {t^{k}(1-t)^{k}\,dt} {({1}- {2t})^{k-2}}
\sum_{j=0}^k (-1)^j\binom kj
\frac {1} {t^{k-j}(1-t)^j-x}.
\label{eq:Qxt}
\end{align}
The remaining arguments are completely analogous to those of the proof
of Theorem~\ref{thm:vert} and are therefore left to the reader.
\end{proof}

\section{The case $k=3$}
\label{sec:k=3}

The case of triangulations of a subdivided triangle, that is, the case
where $k=3$, is particularly interesting from the point of view of
exact enumeration formulas. By \eqref{eq:Formel2}, we know that
\begin{equation} \label{eq:k=3A}
\trkr{3}{r}=
-\sum_{\ell=0} ^{3r-1}  \,2^{\ell-1}\,
 \binom {3r-\ell-5} {3r-\ell-1}
+3\sum_{\ell=0} ^{2r-2}\,2^{\ell-1}\,
 \binom {3r-\ell-5} {2r-\ell-2}
-3\sum_{\ell=0} ^{r-3} \,2^{\ell-1}\,
 \binom {3r-\ell-5} {r-\ell-3}.
\end{equation}
A simpler formula can be obtained if one reads coefficients from the
right-hand side of \eqref{eq:intFormel} in a way that differs from
the one done in the proof of Theorem~\ref{thm:Formel}.
Namely, we write
\begin{align*}
&\trkr{k}{r}
=
-\frac {1} {4\pi i}\int_{\mathcal C}
\frac {dt} {({1}- {2t})}
\left(
t^{-3r}(1-t)^{3}-3t^{-2r+1}(1-t)^{-r+2}
+3t^{-r+2}(1-t)^{-2r+1}
\right)\\
&\quad =
-\frac {1} {4\pi i}\int_{\mathcal C}
\frac {dt} {({1}- {2t})}
t^{-3r}(1-t)^{3}
+\frac {3} {4\pi i}\int_{\mathcal C}
\frac {dt} {({1}- {2t})}
\left(
t^{-2r+1}(1-t)^{-r+2}
-t^{-r+2}(1-t)^{-2r+1}
\right)\\
&\quad =
-\frac {1} {4\pi i}\int_{\mathcal C}
\frac {dt} {({1}- {2t})}
t^{-3r}(1-t)^{3}
+\frac {3} {4\pi i}\int_{\mathcal C}
\sum_{j=0}^r t^{-2r+1+j}(1-t)^{-r+1-j}\,dt.
\end{align*}
The second integral can again be interpreted as a coefficient extraction
formula. In the first integral, we blow up $\mathcal C$ so that it
tends to the circle at infinity. While doing this, we pass over the
pole at $t=1/2$. Hence, the residue at this point must be taken into
account. The integral along the circle at infinity vanishes since
the integrand is of the order $O(t^{-2})$ as $\vert
t\vert\to\infty$.
If this is taken into account, then we obtain the alternative formula
\begin{equation}
\trkr{3}r=
-2^{3r-5}
+\frac {3} {2}
\sum_{j=0}^r \binom {3r-4} {2r-2-j}
=
-2^{3r-5}
+\frac {3} {2}
\sum_{j=0}^r \binom {3r-4} {r-2+j}.
\label{eq:k=3B}
\end{equation}
Making use of the symmetry of binomial
coefficients and of the binomial theorem, it is a simple matter
to verify that the above is equivalent to
\begin{equation}
\trkr{3}r=
{2^{3r-4}}
-3
\sum_{j=0}^{r-3} \binom {3r-4} {j}.
\label{eq:k=3B2}
\end{equation}

We entered the sequence $(\trkr{3}{r})_{r\ge1}$ into
the On-line Encyclopedia of Integer Sequences \cite{oeis}.
This produced the hit OEIS/A087809, which in particular said
that (according to \cite{oeis} a conjecture of Benoit Cloitre)
another (elegant) formula must be
\begin{equation} \label{eq:k=3C}
\trkr{3}{r} = \sum_{i,j,k \geq 0}
\binom{r-1}{i+j}\binom{r-1}{j+k}\binom{r-1}{i+k}.
\end{equation}
We prove this conjecture, in a more general context, in the next
section; see Theorem~\ref{prop:abc_bl}.

There is yet another (substantially) different formula for $\trkr{3}r$.
By computer experiments, utilizing the guessing features of {\sl
  Maple}, we were led to conjecture that
\begin{equation} \label{eq:k=3D}
\trkr{3}{ r+2} =
3\binom{3r+2}{r}+\sum_{j=0}^{r}\frac{5j+1}{2j+1}\binom{3j}{j}8^{r-j}  .
\end{equation}
This formula can be established in the following way.
The (already established) formula \eqref{eq:k=3B2} for $\trkr{3}r$
satisfies the recurrence
\begin{equation} \label{eq:rek}
\trkr{3}{r+1}-8\trkr{3}r
=
\frac {3\,(5 r^2 - 19 r +6)\, (3 r-4)!} {(r-2)! \,(2 r)!}.
\end{equation}
This is easy to see by applying the relation
\[
\binom {3r-1}j=
\binom {3r-4}{j}+3\binom {3r-4}{j-1}+3\binom {3r-4}{j-2}+\binom {3r-4}{j-3}
\]
to the binomial coefficient appearing in the definition of
$\trkr{3}{r+1}$ (or by entering the sum in \eqref{eq:k=3B2} into
the Gosper--Zeilberger algorithm; cf.~\cite{AeqB}).
On the other hand, it is routine to verify that the expression
in \eqref{eq:k=3D} (with $r$ replaced by $r-2$) satisfies the same
recurrence. Comparison of an initial value then completes the proof
of \eqref{eq:k=3D}.

Finally, our results also enable us to establish another conjecture
reported in Entry OEIS/A087809 of \cite{oeis}, namely an expression
for the generating function of the numbers $\trkr{3}{r}$ that is
more compact than the expression produced by Theorem~\ref{thm:hor}
for $k=3$. According to \cite{oeis}, this expression was found by
Mark van Hoeij (presumably) by using his computer algebra tools.
It reads
\begin{equation} \label{eq:tr3GF}
\sum_{r\ge1}\trkr{3}{r-1}x^r=
\frac {10 g^3(x)-17 g^2(x)+7 g(x)-1}
{(1-3 g(x)) (2 g(x)-1) (4 g^2(x)-6 g(x)+1)},
\end{equation}
where $g(x)(1-g(x))^2 = x$.
Indeed, to see this, we first observe that
\[
(2 g(x)-1) (4 g^2(x)-6 g(x)+1)
=8g(x)(1-g(x))^2-1=8x-1.
\]
If we use this in \eqref{eq:tr3GF}, then we see that van Hoeij's claim is
\begin{align}
\notag
\trkr{3}{r+1}&=\coef{x^r}
\frac {1-7 g(x)+17 g^2(x)-10 g^3(x)}
{(1-3 g(x)) (1-8x)}\\
&=\sum_{j=0}^\infty \coef{x^{r-j}} 8^j
\frac {1-7 g(x)+17 g^2(x)-10 g^3(x)}
{(1-3 g(x))}.
\label{eq:Hoeij}
\end{align}
The coefficient of $x^{r-j}$ on the right-hand side is conveniently
computed using the second form of Lagrange inversion (see
\cite[Eq.~(1.2)]{KratAB}). We obtain
\begin{align*}
\coef{x^n} &
\frac {1-7 g(x)+17 g^2(x)-10 g^3(x)}
{(1-3 g(x))}\\
&=\coef{x^{-1}}\frac {1-7 x+17 x^2-10 x^3}
{(1-3 x)}
\big(x(1-x)^2\big)^{-n-1}\frac {d} {dx}\big(x(1-x)^2\big)\\
&=\coef{x^{n}}(1-7 x+17 x^2-10 x^3)\,(1-x)^{-2n-1}
\end{align*}
This is now substituted on the right-hand side of \eqref{eq:Hoeij}.
It yields
\begin{multline*}
\sum_{j=0}^\infty 8^j
\binom {3(r-j)} {r-j}
-7\sum_{j=0}^\infty 8^j
\binom {3(r-j)-1} {r-j-1}
+17
\sum_{j=0}^\infty 8^j
\binom {3(r-j)-2} {r-j-2}
-10
\sum_{j=0}^\infty 8^j
\binom {3(r-j)-3} {r-j-3}\\
=\sum_{j=0}^r 8^{r-j}
\binom {3j} {j}
-7\sum_{j=0}^r 8^{r-j}
\binom {3j-1} {j-1}
+17
\sum_{j=0}^r 8^{r-j}
\binom {3j-2} {j-2}
-10
\sum_{j=0}^r 8^{r-j}
\binom {3j-3} {j-3}.
\end{multline*}
In the first sum, we shift the index by replacing $j$ by $j-1$.
Thus, we obtain
\begin{align*}
\binom {3r}{r}+
\sum_{j=0}^r 8^{r-j} &
\left(8\binom {3j-3} {j-1}
-7
\binom {3j-1} {j-1}
+17
\binom {3j-2} {j-2}
-10
\binom {3j-3} {j-3}\right)\\
&=
\binom {3r}{r}+
\sum_{j=1}^r 8^{r-j}
\frac {5j-4} {2j-1}
\binom {3j-3} {j-1}\\
&=
\binom {3r}{r}+
\sum_{j=0}^{r-1} 8^{r-1-j}
\frac {5j+1} {2j+1}
\binom {3j} {j}.
\end{align*}
By \eqref{eq:k=3D}, this expression equals $\trkr{3}{r+1}$,
which establishes van Hoeij's guess.

\section{The case $k=3$, non-balanced version}\label{sec:abc}

In this section, we generalize two formulas
for $\mathsf{tr}(3, r)$
that we obtained in Section~\ref{sec:k=3}
to the non-balanced case.
The proofs use quite elementary tools
and shed more light on the structure of subdivided triangles.
More precisely, we prove a generalization of~\eqref{eq:k=3C}
by considering a trivariate generating function and subsequently
performing coefficient extraction,
and a generalization of~\eqref{eq:k=3B2}
by partitioning a triangulation of a subdivided triangle
into structural blocks.

\smallskip

First we introduce some notation.
Let $\mathrm{\Delta}(a,b,c)$ be the triangle $ABC$ whose sides are
subdivided as follows:
the side $BC$ is subdivided by $a$ points,
the side $CA$ by $b$ points,
and the side $AB$ by $c$ points.

Let $T$ be a triangulation of $\mathrm{\Delta}(a,b,c)$.
An \emph{ear} is a triangle of $T$ that contains a corner
of $ABC$.
For example, the triangulation in Figure~\ref{fig:dt}(a) has ears
in all three corners (marked in grey colour),
while the triangulation in Figure~\ref{fig:dt}(b)
has ears in the corners $A$ and $B$ (again marked in grey colour),
but none in~$C$.
An \emph{ear diagonal} is the side of an ear that
lies in the interior of $ABC$.
A \emph{central triangle} is a triangle of $T$
whose vertices are interior points of different sides of $ABC$.
For example, the triangulation in Figure~\ref{fig:dt}(a) contains
a central triangle (namely the green triangle), while the triangulation
in Figure~\ref{fig:dt}(b) is one without central triangle.
A \emph{regular triangle} is a triangle of $T$
which is neither an ear nor a central triangle.
A \emph{corner-side diagonal} is a diagonal of $T$
one of whose endpoints
is a corner of $ABC$ and the other an interior
point of the opposite side. Examples of corner-side diagonals are the
red diagonals in the triangulation in Figure~\ref{fig:dt}(b).
On the other hand, the triangulation in Figure~\ref{fig:dt}(a) does
not contain any corner-side diagonal.

It is easy to observe the following facts.
\begin{observation}\label{obs:prop_abc}
Triangulations of $\mathrm{\Delta}(a,b,c)$ have the following properties:
\begin{enumerate}
  \item Each regular triangle shares exactly one edge
  with a side of $ABC$.
  \item
Any triangulation of $\mathrm{\Delta}(a,b,c)$
  has corner-side diagonals emanating from at most one corner.
  \item Any triangulation of $\mathrm{\Delta}(a,b,c)$
  has at most one central triangle.
\end{enumerate}
\end{observation}

More precisely: assume $(a,b,c)\neq(0,0,0)$,
and let $T$ be a triangulation of $\mathrm{\Delta}(a,b,c)$.
Then \emph{either}
$T$ has one central triangle,
three ears,
and no corner-side diagonal,
\emph{or}
$T$ has no central triangle,
two ears,
and at least one corner-side diagonal emanating from the remaining corner.
Triangulations of the former kind will be called
\emph{$\mathrm{T}$-triangulations} (see Figure~\ref{fig:dt}(a) for
an example),
and triangulations of the latter kind will be called
\emph{$\mathrm{D}$-triangulations} (see Figure~\ref{fig:dt}(b) for an example).
Moreover,
a \emph{$\mathrm{D}_A$-triangulation} is a
($\mathrm{D}$-)triangulation
that contains a corner-side diagonal one of whose endpoints is $A$,
and $\mathrm{D}_B$-
and $\mathrm{D}_C$-triangulations are similarly defined.
The triangulation in Figure~\ref{fig:dt}(b) is a $\mathrm{D}_C$-triangulation.

\begin{figure}
\begin{center}
\includegraphics[scale=0.75]{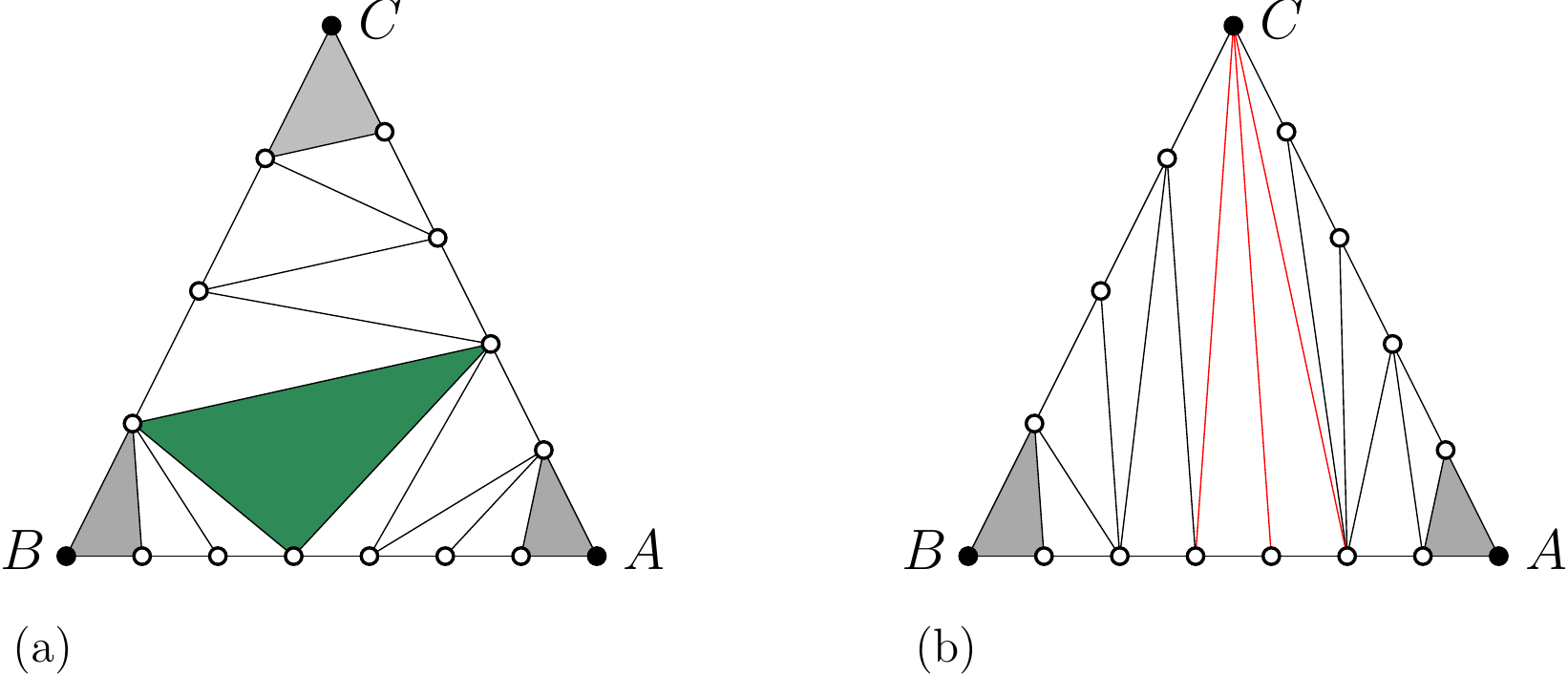}
\end{center}
\caption{Two triangulations of $\Delta(3,4,6)$:
(a) a $\mathrm{T}$-triangulation;
(b) a $\mathrm{D}_C$-triangulation.}
\label{fig:dt}
\end{figure}
Denote the sets of
$\mathrm{T}$-,
$\mathrm{D}$-,
$\mathrm{D}_A$-,
$\mathrm{D}_B$-,
and $\mathrm{D}_C$-triangulations
of \tabc\break
by
$\mathsf{TR}_{\mathrm{T}}(\Delta(a,b,c))$,
$\mathsf{TR}_{\mathrm{D}}(\Delta(a,b,c))$,
$\mathsf{TR}_{\mathrm{D}_A}(\Delta(a,b,c))$,
$\mathsf{TR}_{\mathrm{D}_B}(\Delta(a,b,c))$, and\break
$\mathsf{TR}_{\mathrm{D}_C}(\Delta(a,b,c))$,
respectively.
Similarly, denote their cardinalities by
$\mathsf{tr}$ with appropriate specification:
$\mathsf{tr}_{\mathrm{T}}(\Delta(a,b,c))$, etc.

The theorem below summarizes our counting formulas for the various
classes of triangulations that we just defined. In particular,
it provides the promised generalization of~\eqref{eq:k=3B2}
in \eqref{eq:abc_total}.

\begin{theorem}\label{prop:abc_all}
For any non-negative integers $a, b, c$
not all equal to zero,
\begin{enumerate}
  \item
  the number of\/ $\mathrm{D}$-triangulations of $\Delta(a,b,c)$ is
  \begin{equation}\label{eq:abc_d}
  \mathsf{tr}_{\mathrm{D}}(\Delta(a,b,c))=\binom{a+b+c-1}{a-1}
+\binom{a+b+c-1}{b-1}+\binom{a+b+c-1}{c-1};
  \end{equation}
  \item
  the number of\/ $\mathrm{T}$-triangulations of $\Delta(a,b,c)$ is
  \begin{equation}
 \mathsf{tr}_{\mathrm{T}}(\Delta(a,b,c))
 = 2^{a+b+c-1}
- \sum_{\ell=0}^{a-1}\binom{a+b+c-1}{\ell}-
\sum_{\ell=0}^{b-1}\binom{a+b+c-1}{\ell}-
\sum_{\ell=0}^{c-1}\binom{a+b+c-1}{\ell};
\label{eq:abc_t}
\end{equation}
  \item
  the total number of triangulations of $(\Delta(a,b,c)$ is
  \begin{equation}
 \mathsf{tr}(\Delta(a,b,c))
 = 2^{a+b+c-1}
- \sum_{\ell=0}^{a-2}\binom{a+b+c-1}{\ell}-
\sum_{\ell=0}^{b-2}\binom{a+b+c-1}{\ell}-
\sum_{\ell=0}^{c-2}\binom{a+b+c-1}{\ell}.
\label{eq:abc_total}
\end{equation}
\end{enumerate}
\end{theorem}

\begin{proof} (1)
We first show that
\begin{equation}\label{eq:abc_d0}
\mathsf{tr}_{\mathrm{D}_A}(\Delta(a,b,c))=\binom{a+b+c-1}{a-1}.
\end{equation}

In order to see that, consider $T$,
a $\mathrm{D}_A$-triangulation of $\Delta(a,b,c)$.
The triangles of $T$ can be linearly ordered as follows.
Consider the directed segment $CB$, and
shift it slightly
(``infinitesimally'') into the interior of $ABC$.
The segment obtained in this way intersects all the triangles of $T$
and, thus, induces a linear order on them.

By Observation~\ref{obs:prop_abc}(1),
each regular triangle of $T$
shares exactly one edge with one of the sides of $ABC$.
We encode the regular triangles that share an edge with $CB$ by $0$,
and those that share an edge with $CA$ or with $AB$ by $1$.
Using the linear order that was described above,
we obtain a $\{0,1\}$-sequence of length $a+b+c-1$,
in which $0$ occurs $a-1$ times
and
$1$ occurs $b+c$ times.
See Figure~\ref{fig:dtlin} for an illustration.
It is easy to see that this correspondence between
$\mathrm{D}_A$-triangulations of $\Delta(a,b,c)$
and $\{0,1\}$-sequences with
$a-1$ occurrences of $0$ and $b+c$ occurrences of $1$
is bijective.
(In particular, since $b$ and $c$ are fixed,
it is determined uniquely whether a triangle encoded by $1$ shares an edge with
$CA$ or with $AB$.)
Since the number of such sequences is $\binom{a+b+c-1}{a-1}$,
we obtain~\eqref{eq:abc_d0}.
Finally, due to symmetry, we get~\eqref{eq:abc_d}.
\begin{figure}
\begin{center}
\includegraphics[scale=0.75]{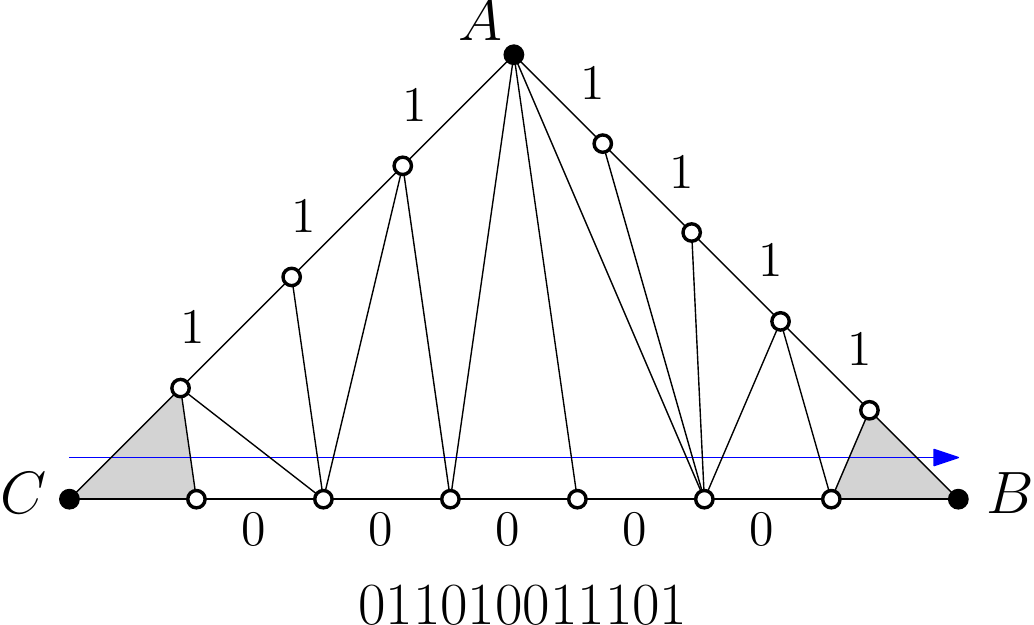}
\end{center}
\caption{Illustration for the proof of Theorem~\ref{prop:abc_all}.1.}
\label{fig:dtlin}
\end{figure}

\begin{remark} A special case of ~\eqref{eq:abc_d0},
the formula $\mathsf{tr}(\Delta(a,b,0))=\binom{a+b}{a}$,
was already mentioned in~\cite{hn}.
\end{remark}

(2) Now we derive the formula
\eqref{eq:abc_t}
for the number of $\mathrm{T}$-triangulations
of $\Delta(a,b,c)$.
By definition and by Observation~\ref{obs:prop_abc}(3),
any $\mathrm{T}$-triangulation $T$ of $\Delta(a,b,c)$ has a
unique central triangle.
If we remove the central triangle from $T$,
then $T$ decomposes into three triangulations:
a triangulation of $\Delta(a_2,b_1,0)$,
a triangulation of $\Delta(b_2,c_1,0)$, and
a triangulation of $\Delta(c_2,a_1,0)$,
where $a_1+a_2 = a-1$, $b_1+b_2 = b-1$, $c_1+c_2 = c-1$.
Conversely, each (appropriately combined) triple of such triangulations
generates a $\mathrm{T}$-triangulation of $\Delta(a,b,c)$.
Since, as mentioned above, we have $\Delta(a,b,0)=\binom{a+b}{a}$,
and since $\frac{1}{1-x-y}$ is the bivariate generating function
for the array
$\left(\binom{a+b}{a}\right)_{a, b \geq 0}$,
we conclude that $\frac{xyz}{(1-x-y)(1-y-z)(1-z-x)}$
is the trivariate generating function for
$(\mathsf{tr}_{\mathrm{T}}(\Delta(a,b,c)))_{a, b, c \geq 0}$.
To be precise, for each fixed triple $(a,b,c)$, we have
\begin{equation}\label{eq:tri}
\mathsf{tr}_{\mathrm{T}}(\Delta(a,b,c)) = [x^a y^b z^c] \,
\frac{xyz}{(1-x-y)(1-y-z)(1-z-x)}.
\end{equation}

In order to extract the coefficients, we ignore the factor $xyz$ in
the numerator for a while. We have
\begin{align}
 [x^a y^b z^c] \, \frac{1}{(1-x-y)(1-y-z)(1-z-x)}  \notag
& = \sum_{i=0}^{a} \sum_{j=0}^{b} \left( \binom{i+j}{i}  \cdot
 \sum_{k=0}^{c} \binom{b-j+k}{b-j} \binom{a-i+c-k}{a-i} \right) \notag
 \\
& = \sum_{i=0}^{a} \sum_{j=0}^{b} \binom{i+j}{i}
 \binom{a+b+c+1-i-j}{a+b+1-i-j}  \notag \\
& = \sum_{i=0}^{a} \sum_{j=0}^{b} \binom{i+j}{i}  \binom{a+b+c+1-i-j}{c}.
\label{eq:abc1}
\end{align}
For the second equality we used the standard combinatorial identity
\[
\sum_{i=0}^{\ell} \binom{m+i}{m} \binom{n+\ell-i}{n} =
\binom{m+n+\ell+1}{m+n+1},
\]
which is a special instance of Chu--Vandermonde summation.
We may use it again in order to evaluate
the inner sum of the remaining double sum,
for $0 \leq j \leq a+b+1-i$ rather than $0 \leq j \leq b$:
\begin{equation}\label{eq:abc2}
\sum_{j=0}^{a+b+1-i} \binom{i+j}{i}  \binom{a+b+c+1-i-j}{c} =
\binom{a+b+c+2}{c+1+i}.
\end{equation}
Now we continue simplifying \eqref{eq:abc1}.
We use \eqref{eq:abc2} and subtract the extra terms
which also have this form (up to an interchange of the summations over
$i$ and $j$).
Writing $s=a+b+c+2$, we have
\begin{align*}
 \sum_{i=0}^{a} \sum_{j=0}^{b}& \binom{i+j}{i}  \binom{a+b+c+1-i-j}{c}    \\
& = \sum_{i=0}^{a} \sum_{j=0}^{a+b+1-i} \binom{i+j}{i}
  \binom{a+b+c+1-i-j}{c}   \ - \
\sum_{j=b+1}^{a+b+1} \sum_{i=0}^{a+b+1-j} \binom{i+j}{i}
\binom{a+b+c+1-i-j}{c}  \\
& =\sum_{i=0}^{a}\binom{s}{c+1+i} - \sum_{j=b+1}^{a+b+1} \binom{s}{c+1+j}
\ = \ \sum_{\ell=c+1}^{a+c+1}\binom{s}{\ell} -
\sum_{\ell=b+c+2}^{a+b+c+2} \binom{s}{\ell} \ = \
\sum_{\ell=c+1}^{a+c+1}\binom{s}{\ell} - \sum_{\ell=0}^{a}
\binom{s}{\ell}  \\
& =\sum_{\ell=0}^{s}\binom{s}{\ell} - \sum_{\ell=0}^{a}
\binom{s}{\ell} - \sum_{\ell=0}^{c} \binom{s}{\ell} -
\sum_{\ell=a+c+2}^{s} \binom{s}{\ell} \ = \ 2^s - \sum_{\ell=0}^{a}
\binom{s}{\ell} - \sum_{\ell=0}^{b} \binom{s}{\ell} -
\sum_{\ell=0}^{c} \binom{s}{\ell}. \end{align*}
Taking into account the factor $xyz$ in \eqref{eq:tri}, we obtain
\eqref{eq:abc_t}.

\medskip
(3) Finally, we obtain \eqref{eq:abc_total} by adding
\eqref{eq:abc_d} and \eqref{eq:abc_t}.
\end{proof}

\begin{remarks}(1) For certain specific choices of parameters,
formulas that can be further simplified can be obtained.
For example, we have $\mathsf{tr}_{\mathrm{T}}(\Delta(a,b,1))=\binom{a+b}{a}-1$.
Recall that $\mathsf{tr}(\Delta(a,b,0))=\binom{a+b}{a}$.
We leave it as an exercise for the reader to find a (simple)
``almost bijection'' between $\mathsf{TR}_{\mathrm{T}}(\Delta(a,b,1))$
and $\mathsf{TR}(\Delta(a,b,0))$.

\medskip
(2) Item (1) of
Theorem~\ref{prop:abc_all}
can also be proven in a way similar to our proof of
Item~(2) ---
by considering a trivariate generating function and extracting coefficients.
Doing this, we obtain
$\mathsf{tr}_{\mathrm{D}_A}(\Delta(a,b,c)) = [x^a y^b z^c]
\frac{xyz}{(1-x)(1-x-y)(1-x-z)}$,
and similarly for
$\mathsf{tr}_{\mathrm{D}_B}(\Delta(a,b,c))$ and $\mathsf{tr}_{\mathrm{D}_C}(\Delta(a,b,c))$.
\end{remarks}

Next we prove the announced generalization of Formula~\eqref{eq:k=3C}
to the non-balanced case.

\def\aa{\alpha}
\def\bb{\beta}
\def\gg{\gamma}
\begin{theorem}\label{prop:abc_bl}
For any non-negative integers
$a,b,c$,
we have
\begin{equation}\label{eq:abc_bl}
\mathsf{tr}(\Delta(a,b,c)) =
\sum_{\aa, \bb, \gg \geq 0}
\binom{a}{\aa+\bb}
\binom{b}{\bb+\gg}
\binom{c}{\gg+\aa}.
\end{equation}
\end{theorem}

\begin{proof}
We use a uniform notation
similarly to the notation that we used for the balanced case
(see Figure~\ref{fig:abc_bl}).
We denote the corners of the triangle by
$P_0=P_{0,0}$, $P_1=P_{1,0}$, $P_2=P_{2,0}$ (say, clockwise),
with arithmetic $\mathrm{mod}$~$3$ in the first index.
For each $i \in \{0,1,2\}$, the side $P_{i} P_{i+1}$
is subdivided by $s_i$ points
$P_{i, 1}, P_{i, 2}, \dots, P_{i, s_i}$
(in the direction from $P_{i}$ to $P_{i+1}$).
Moreover, we set
$P_{i, s_i+1} = P_{i+1}$.
\begin{figure}
\begin{center}
\includegraphics[scale=0.9]{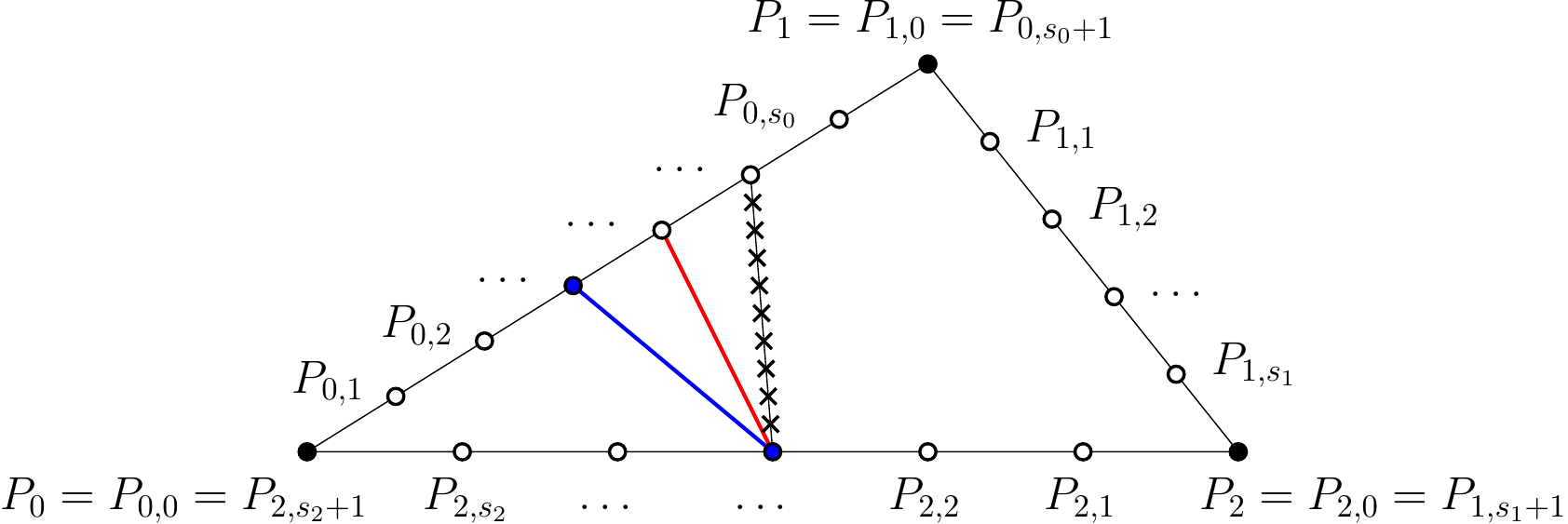}
\end{center}
\caption{Illustration for the proof of Theorem~\ref{prop:abc_bl}:
notation and definition of $F_T$.
The diagonals shown in blue and red belong to $T$;
the diagonal shown by crosses does not belong to $T$.
Hence, the blue diagonal belongs to $F_T$.}
\label{fig:abc_bl}
\end{figure}
In this notation, Formula~\eqref{eq:abc_bl} reads
\begin{equation} \label{eq:abc_bl1}
\mathsf{tr}(\Delta(s_0,s_1,s_2)) =
\sum_{\aa_1, \aa_2, \aa_3 \geq 0}
\binom{s_0}{\aa_0+\aa_1}
\binom{s_1}{\aa_1+\aa_2}
\binom{s_2}{\aa_2+\aa_3}.
\end{equation}

Let $F$ be some (possibly empty)
set of diagonals of $\Delta(s_0,s_1,s_2)$
which connect \textbf{interior} points of two sides of the basic triangle
(that is, $F$ does not contain corner-side diagonals),
and which are pairwise disjoint
(that is, they are not only non-crossing but also do not share endpoints).
Such sets will be called \textit{fundamental sets} (of diagonals of $\Delta(s_0,s_1,s_2)$).
Each diagonal in a fundamental set $F$ can be uniquely represented as
$P_{i-1, \ell} P_{i, m}$ for some $i\in\{0,1,2\}$,
$1 \leq \ell \leq s_{i-1}$,
$1 \leq m \leq s_{i}$.
We say that this diagonal
\emph{separates} the corner $P_{i}$.

We say that a fundamental set $F$ has {\it type} $(\aa_0, \aa_1, \aa_2)$
if, for $i \in \{0,1,2\}$,
the number of elements of $F$ that separate the corner $P_i$ is exactly $\aa_i$.
Notice that $F$ is uniquely determined by the set of the endpoints of its elements.
Indeed, if, for $i\in\{0,1,2\}$,
exactly $\bb_i$ endpoints of the elements of $F$ lie on $P_iP_{i+1}$,
then the type of $F$ is $(\aa_0, \aa_1, \aa_2)$, where
$\aa_i = (\bb_{i-1}+\bb_{i} - \bb_{i+1})/2$.
Once we know the set of endpoints of the elements of $F$ and its
type, the elements of $F$ themselves can be identified at once.
It follows that the number of fundamental sets of type $(\aa_0, \aa_1, \aa_2)$
is $\binom{s_0}{\aa_0+\aa_1}
\binom{s_1}{\aa_1+\aa_2}
\binom{s_2}{\aa_2+\aa_3}$,
and the total number of fundamental sets is precisely the right-hand side of
\eqref{eq:abc_bl1}.
Thus, in order to prove the claim, it suffices to find a bijection
between the set of triangulations of
$\Delta(s_0,s_1,s_2)$
and the set of its fundamental sets.

Let $T$ be a triangulation of $\Delta(s_0,s_1,s_2)$. We define
\[
F_T := \left\{
\begin{array}{cl}
  P_{i-1, \ell} P_{i, m}\colon  & i\in\{0,1,2\}, \
1 \leq \ell \leq s_{i-1}, \
1 \leq m \leq s_{i};
 \\
   & P_{i-1, \ell} P_{i, m} \in T, \
P_{i-1, \ell} P_{i, m+1} \in T, \
P_{i-1, \ell} P_{i, m+2} \not \in T
\end{array}
\right\}.
\]
(Notice that, if $m = s_{i}$,
then $P_{i-1, \ell} P_{i, m+1}$ is a corner-side diagonal,
and the last condition, $P_{i-1, \ell} P_{i, m+2} \not \in T$, is satisfied automatically.)
Figure~\ref{fig:abc_bl} illustrates this definition:
the diagonal coloured blue satisfies the just described condition
and, therefore, is an element of $T_F$.

It is easy to verify that $F_T$ is a fundamental set.
Moreover, next we show that,
given a fundamental set $F$, there is a unique triangulation $T$
such that $F_T=F$.
This triangulation $T$ can be reconstructed from $F$ by applying the following procedure.

Given $F$, we define another set of diagonals
(a \textit{modified fundamental set}), by
  \[
  F'=\{P_{i-1, \ell} P_{i, m+1} \colon \ \ P_{i-1, \ell} P_{i, m} \in F\}.
  \]
In addition, for each corner $P_i$ such that $F'$ contains no corner-side diagonal
  one of whose endpoints is $P_i$, we
  add the ear diagonal $P_{i-1, s_{i-1}} P_{i, 1}$ to $F'$.
  See Figure~\ref{fig:abc_bl_rule}(a):
  a ``generic'' element of $F$ is coloured blue,
  the corresponding element of $F'$ is coloured red;
  another diagonal is coloured red because it is an ear diagonal.
\begin{figure}
\begin{center}
\includegraphics[scale=0.95]{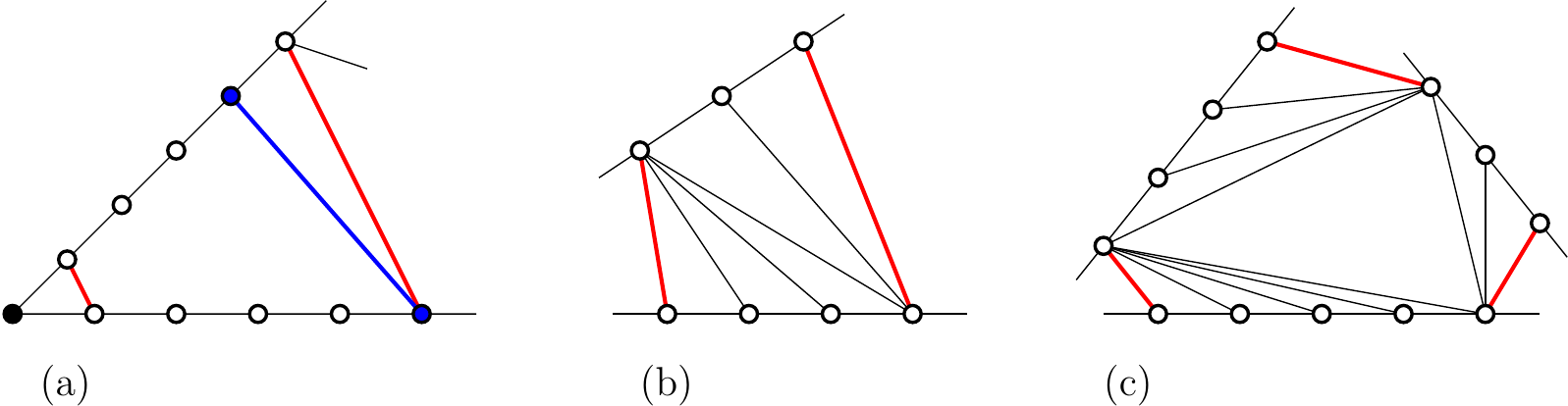}
\end{center}
\caption{Rules for reconstructing $T$ from $F=F_T$.
Blue diagonals are the elements of $F$.
Red diagonals are the elements of $F'$.
(a) Definition of $F'$.
(b) Triangulation of a block bounded by two elements of $F'$.
(c) Triangulation of a block bounded by three elements of $F'$.
}
\label{fig:abc_bl_rule}
\end{figure}

\begin{figure}
\begin{center}
\includegraphics[scale=0.9]{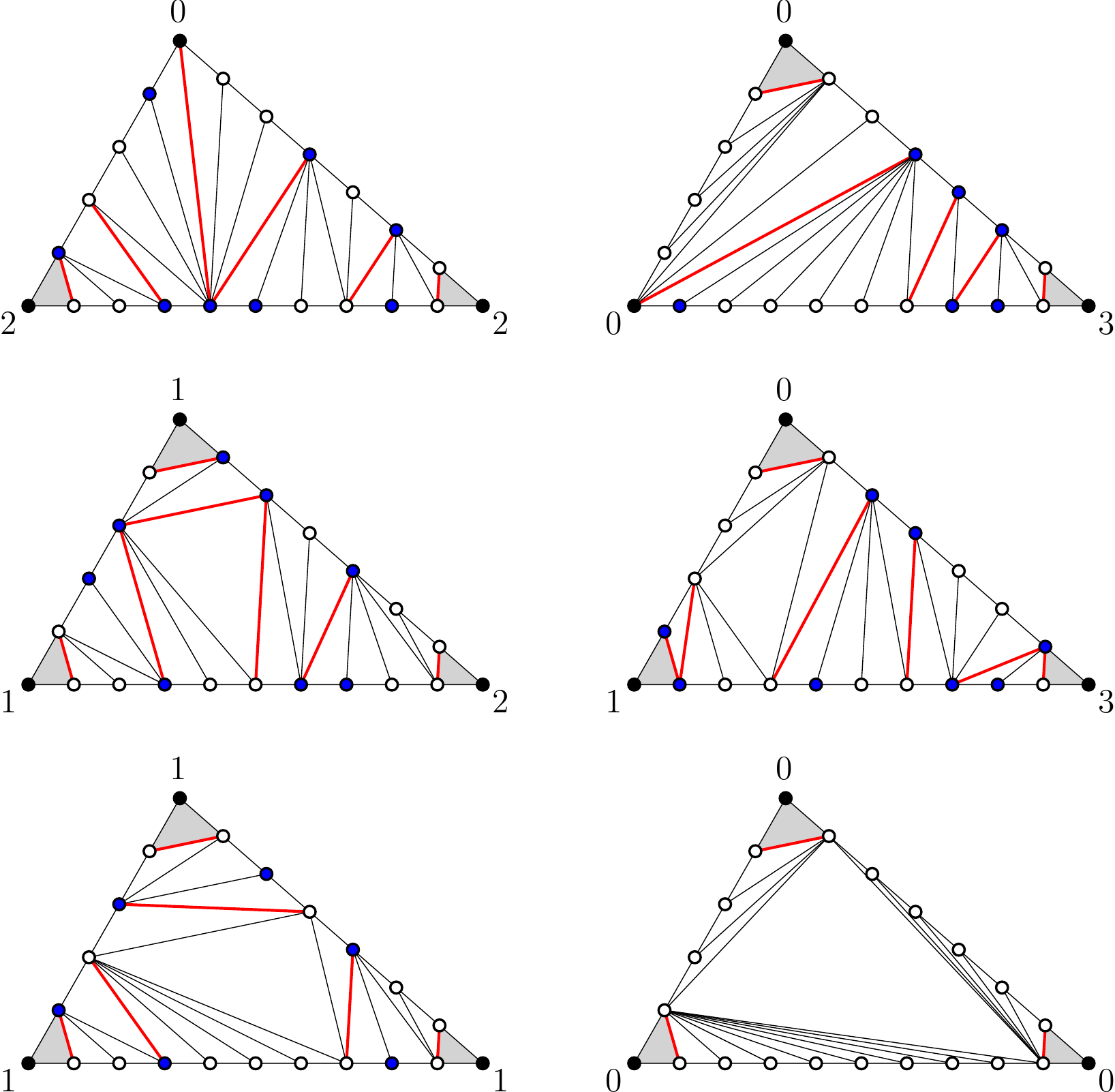}
\end{center}
\caption{Reconstructing $T$ from $F=F_T$.
Blue points are the endpoints of the elements of $F$.
Red diagonals are the elements of $F'$.
The numbers at the corners are $\aa_0$, $\aa_1$ and $\aa_2$.
}
\label{fig:abc_bl_ex}
\end{figure}

The elements of $F'$ are not necessarily disjoint ---
  they can share endpoints, --- but still they are non-crossing.
  Therefore they partition $\Delta(s_0,s_1,s_2)$ into several parts
  that we call \emph{blocks}.
  The boundary of each block contains at most three elements of $F'$
  (in fact, we have two or three ears whose boundaries contain exactly one element of $F'$,
  at most one block whose boundary contains three elements of $F'$,
  and all other blocks whose boundaries contain exactly two elements of $F'$).

Then we complete $F'$ to a triangulation of $\Delta(s_0,s_1,s_2)$
by triangulating the blocks according to the following rules:
\begin{itemize}
  \item Suppose $B$ is a block whose boundary contains exactly two elements of $F'$:
  $P_{i-1, \ell'}P_{i, m}$ and $P_{i-1, \ell}P_{i, m'}$, where
  $i \in \{0,1,2\}$,
   $0 \leq \ell \leq \ell' \leq s_{i-1}$,
  $1 \leq m \leq m' \leq s_{i}+1$.
  Then we add the diagonal $P_{i-1, \ell}P_{i, m}$
  (unless it belongs to $F'$, which would happen if we have
  $\ell=\ell'$ or $m=m'$).
  At this point there is only one way to complete the triangulation of $B$.
  See Figure~\ref{fig:abc_bl_rule}(b).
  \item Suppose $B$ is a block whose boundary contains three elements of $F'$:
  $P_{i-1, \ell'}P_{i, m}$, $P_{i, m'}P_{i+1, p}$, and $P_{i+1, p'}P_{i+1, \ell}$, where
  $i \in \{0,1,2\}$,
   $1 \leq \ell \leq \ell' \leq s_{i-1}$,
$1 \leq m \leq m' \leq s_{i}$,
$1 \leq p \leq p' \leq s_{i+1}$
Then we add three diagonals
(or, more precisely: those of them that do not belong to $F'$)
that form the triangle
$P_{i-1, \ell}P_{i, m}P_{i+1, p}$.
  At this point there is only one way to complete the triangulation of $B$.
  See Figure~\ref{fig:abc_bl_rule}(c).
\end{itemize}

Once this is done for all blocks, we have a triangulation $T$ of $\Delta(s_0,s_1,s_2)$.
It is routine to verify that $T$
contains all the elements of $F$,
and that $T$ is the unique triangulation of $\Delta(s_0,s_1,s_2)$
such that $F_T = F$.
See Figure~\ref{fig:abc_bl_ex} for some examples.

We established a bijection between the set of triangulations of
$\Delta(s_0,s_1,s_2)$
and the set of its fundamental sets. As explained above,
this completes the proof of the claim.

To summarize:
while \emph{fundamental sets} are clearly enumerated
by the right-hand side of \eqref{eq:abc_bl},
it is \emph{modified fundamental sets} that describe
a very natural structural decomposition of triangulations into blocks.
\end{proof}

\section{Asymptotics}\label{sec:asy}

Here, we determine the asymptotic behaviour of $\trkr{k}{r}$.
Our starting point is another integral representation of
$\trkr{k}{r}$. It is motivated by the fact that the integrand
in \eqref{eq:intFormel}, $I_{r,k}(t)$ say,
has one saddle point at $t=1/2$ for large
$k$ and/or~$r$, which is easily verified by solving the saddle
point equation $\frac {d} {dt}I_{r,k}(t)=0$ for large $k$ and/or~$r$.%
\footnote{
Strictly speaking, the point $t=1/2$ is not a saddle point
of the function $t\to \vert I_{r,k}(t)\vert$, since its value at
$t=1/2$ vanishes, that is, $I_{r,k}(1/2)=0$. However, this is ``just"
caused by the factor $(1-2t)^2$ in the numerator (the factor $(1-2t)^k$
in the denominator cancels with $\big((1-t)^{r+1}-t^{r+1}\big)^k$ in the
numerator). If we would ignore the factor $(1-2t)^2$, that is, if we
would instead
consider $I_{r,k}(t)/(1-2t)^2$, then $t=1/2$ is a true saddle point.
So, ``morally," the point $t=1/2$ {\it is} a saddle point of $t\to
\vert I_{r,k}(t)\vert$,
in the sense that the main contribution to the integral comes from a small
environment around $t=1/2$. The ``only" effect of the factor
$(1-2t)^2$ is to lower the polynomial factor in the asymptotic
approximation, while the exponential growth is not affected.}
(The subsequent arguments can however be followed without that
observation.)

\begin{proposition} \label{prop:saddle}
For all positive integers $k$ and $r$ with $rk\ge3$, we have
\begin{equation} \label{eq:intFormel2}
\trkr{k}{r}=
-\frac {2^{(r-2)k}} {\pi }\int_{-\infty} ^\infty
\frac {du} {(1+4u^2)^{rk}
(iu)^{k-2}}
\left(
\Big(1+2iu\Big)^{r+1}
-\Big(1-2iu\Big)^{r+1}
\right)^k.
\end{equation}
\end{proposition}

\begin{proof}
We start with the integral representation \eqref{eq:intFormel}.
We deform the contour $\mathcal C$ so that it passes through the point
$t=1/2$. More precisely, we consider
the family of contours
\begin{equation} \label{eq:contour}
\left\{t:\Re(t)=\tfrac {1} {2}\text{ and }\vert\Im(t)\vert\le \rho\right\}\cup
\left\{t:\vert t-\tfrac {1} {2}\vert=\rho\text{ and }\Re(t)\le \tfrac {1}
        {2}\right\},
\end{equation}
parametrized by positive real numbers $\rho\ge1$, which are supposed to
be oriented in positive direction. In other words, these contours
consist of a vertical straight line segment of length $2\rho$
whose midpoint is $1/2$, and
the left half-circle whose diameter is this very segment.
The integral over these contours still equals $\trkr{k}{r}$
since $t=1/2$ is a
removable singularity of the integrand.

Now we let $\rho\to\infty$. As we already observed in the proof of
Proposition~\ref{prop:Formel}, the integrand is of the order
$O(t^{-2})$ as $\vert t\vert\to\infty$
under our assumptions. Consequently, the integral
over the circle segment of the contour \eqref{eq:contour} will
tend to zero as $\rho\to\infty$. Thus, the number $\trkr{k}{r}$ equals the
integral over the straight line $\{t:\Re(t)=1/2\}$. If we
set $t=\frac {1} {2}+iu$ in \eqref{eq:intFormel}, then we obtain
\eqref{eq:intFormel2} after little rearrangement.
\end{proof}

The integral representation in Proposition~\ref{prop:saddle} now
allows for a convenient asymptotic analysis of $\trkr{k}{r}$.
We distinguish between two scenarios: (1) the number $k$
of corners is fixed, while the number of subdivisions $r$ tends to
infinity;
(2) $k$ tends to infinity, leaving it open whether $r$ remains
fixed or not.

\begin{theorem} \label{thm:rinf}
For fixed $k\ge3$, we have
\begin{equation} \label{eq:rinf}
\trkr{k}{r}=
\frac {2^{(r-1)k}r^{k-3}} {\pi }
\left(\int_{-\infty} ^\infty
\frac {du} {u^{k-2}}
\sin^k(2u)\right)\Big(1+o(1)\Big),
\quad \quad \text{as }r\to\infty.
\end{equation}
\end{theorem}

\begin{proof}
We start with the integral representation \eqref{eq:intFormel2},
in which we make the substitution $u\to u/r$. This leads to
\begin{align*}
\trkr{k}{r}=
-\frac {2^{(r-2)k}\,r^{k-3}} {\pi }\int_{-\infty} ^\infty
\frac {du} {\left(1+\frac {4u^2} {r^2}\right)^{rk}
(iu)^{k-2}}
\left(
\Big(1+\tfrac {2iu} r\Big)^{r+1}
-\Big(1-\tfrac {2iu} r\Big)^{r+1}
\right)^k.
\end{align*}
Making use of dominated convergence,
we may now compute the limit of the above integral as $r\to\infty$,
\begin{align*}
\lim_{r\to\infty}\int_{-\infty} ^\infty
\frac {du} {\left(1+\frac {4u^2} {r^2}\right)^{rk}
(iu)^{k-2}}
\left(
\Big(1+\tfrac {2iu} r\Big)^{r+1}
-\Big(1-\tfrac {2iu} r\Big)^{r+1}
\right)^k
&=
\int_{-\infty} ^\infty
\frac {du} {
(iu)^{k-2}}
\left(
e^{2iu}-e^{-2iu}
\right)^k\\
&=
-2^k\int_{-\infty} ^\infty
\frac {du} {
u^{k-2}}
\sin^k(2u).
\end{align*}
The assertion of the theorem follows immediately.
\end{proof}

\begin{remark}
It is well-known that the integral in \eqref{eq:rinf} can be evaluated
for any specific~$k$, and it equals some rational multiple of $\pi$.
More precisely (cf.\ \cite[333.17]{GrHoAA} or \cite[3.821.12]{GrRyAA}),
the relations
\begin{align} \label{eq:intsinx}
\int_0^\infty \frac {\sin^\lambda(x)} {x^k}\,dx
&=\frac {\lambda} {k-1}
\int_0^\infty \frac {\sin^{\lambda-1}(x)\,\cos(x)} {x^{k-1}}\,dx,
\quad \quad \text{for }\lambda>k-1>0,\\
&=\frac {\lambda(\lambda-1)} {(k-1)(k-2)}
\int_0^\infty \frac {\sin^{\lambda-2}(x)} {x^{k-2}}\,dx
-\frac {\lambda^2} {(k-1)(k-2)}
\int_0^\infty \frac {\sin^{\lambda}(x)} {x^{k-2}}\,dx,
\notag
\\
&\kern6cm
\text{for }\lambda>k-1>1,
\end{align}
together with the ``initial conditions"
(cf.\ \cite[333.14, 333.15]{GrHoAA} or \cite[3.821.7, 3.832.15]{GrRyAA})
\begin{equation} \label{eq:int1}
\int_{-\infty} ^\infty \frac {\sin^{2k-1}(x)} {x}\,dx
=\frac {\sqrt\pi\,\Gamma(k-\frac {1} {2})} {\Gamma(k)}.
\end{equation}
and
\begin{equation} \label{eq:int2}
\int_{-\infty} ^\infty \frac {\sin^{2k-1}(x)\,\cos(x)} {x}\,dx
=\frac {\sqrt\pi\,\Gamma(k-\frac {1} {2})} {2\,\Gamma(k+1)},
\end{equation}
allow for the recursive computation of the integral in \eqref{eq:rinf}
for any specific~$k$. ({\sl Maple} and {\sl Mathematica} know about this.)
\end{remark}

\begin{theorem} \label{thm:kinf}
We have
\begin{equation} \label{eq:kinf}
\trkr{k}{r} = \frac {\big(2^r(r+1)\big)^k} {16\sqrt\pi (r(r+5)/6)^{3/2}k^{3/2}}
\big(1+o(1)\big),
\quad \quad \text{as }k\to\infty,
\end{equation}
where $r$ may or may not stay fixed.
\end{theorem}

\begin{proof}
We start again with the integral representation \eqref{eq:intFormel2}.
Here we do the substitution
$u\to u/\sqrt{kR}$, where $R$ is short for $r(r+5)/6$.
Thereby we obtain
\begin{multline} \label{eq:kinf2}
\trkr{k}{r}
=
\frac {2^{2rk-(r+1)k}} {(kR)^{3/2}}\frac {1} {\pi }\int_{-\infty} ^\infty
\frac {u^2\,du} {(1+\frac {4u^2} {kR})^{rk}
(2iu/(kR)^{1/2})^{k}}\\
\cdot
\left(
\Big(1+\tfrac {2iu} {(kR)^{1/2}}\Big)^{r+1}
-\Big(1-\tfrac {2iu} {(kR)^{1/2}}\Big)^{r+1}
\right)^k.
\end{multline}
Once again, by dominated convergence, we may approximate the
above integral as $k\to\infty$,
\begin{align*}
\int_{-\infty} ^\infty
\frac {u^2\,du} {(1+\frac {4u^2} {kR})^{rk}
(2iu/(kR)^{1/2})^{k}}
&\left(
\Big(1+\tfrac {2iu} {(kR)^{1/2}}\Big)^{r+1}
-\Big(1-\tfrac {2iu} {(kR)^{1/2}}\Big)^{r+1}
\right)^k\\
&\kern-2cm
=
2^{k}\,(r+1)^k
\left(\int_{-\infty} ^\infty
\frac {u^2\,du} {\exp(4u^2r/R)
}
\exp\left(
\frac {r(r-1)} 6 \frac {(2iu)^2} {R}
\right)
\right)
\Big(1+o(1)\Big)\\
&\kern-2cm
=
2^{k}\,(r+1)^k\left(
\int_{-\infty} ^\infty
u^2\,e^{-4u^2}\,du\right)
\Big(1+o(1)\Big)\\
&\kern-2cm
=
2^{k}\,(r+1)^k
\frac {\sqrt\pi} {16}
\big(1+o(1)\big),
\end{align*}
as $k\to\infty.$
If this is substituted back in \eqref{eq:kinf2}, one obtains
\eqref{eq:kinf}.
\end{proof}

\section{Generalizations of the double circle
and their triangulations}\label{sec:dc}

The present research was initially motivated
by the following
open problem from computational geometry:
what is the minimum number of triangulations that a planar set of $n$ points
{in general position}\footnote{
\textit{General position} means that no three points lie on the same line.}
can have, and for which set(s) is this minimum attained?

This is one instance of the research direction
concerning the minimum and the maximum number of plane geometric non-crossing graphs of various kinds, with respect to the number of points.
One typically fixes some naturally defined class $\mathcal{C}$
of such geometric graphs
(for example, triangulations, spanning trees, perfect matchings, etc.),
and asks for the minimum or the maximum number of graphs from $\mathcal{C}$
that a planar set of $n$ points in general position (playing the role of
the vertex set)
can have, and for a characterization of point set(s) on which these
extremal values are attained.
To our knowledge, in all such cases 
no exact results concerning \textbf{maximum} were found
except for trivialities),
but rather lower and upper bounds, usually with substantial gaps
(see~\cite{adam} for a summary of some results of this type).
In contrast,
for many natural families of plane graphs,
the \textbf{minimum}
is attained for sets in convex position:
Aichholzer et al.~\cite{aich} proved that this is the case
for any class of acyclic graphs
(thus, for spanning trees, forests, perfect matchings, etc.\footnote{
For some of these families it was proven earlier by other authors,
but Aichholzer et al.\ gave a unified proof.}),
as well as for the family of all plane graphs,
and that of all connected plane graphs.
However, this is not the case for triangulations:
in~\cite{ahn}, Aichholzer, Hurtado and Noy presented a configuration,
which they called \textit{double circle}, and
which has less triangulations than sets
of the same size (that is, with the same number of points)
in convex position.
Indeed, as was shown by Santos and Seidel in~\cite{santos},
the double circle of size $n$ has $\Theta^*(\sqrt{12}^{\,n})$ triangulations.
It was proven by exhaustive computations~\cite{ak, aich16}
that, for $n \leq 15$,
(only) the double circle of size $n$
has the minimal number of triangulations
over all point sets of size $n$ in general position.
Therefore it was conjectured in~\cite{ahn} that (only)
the double circle minimizes the number of triangulations for any $n$.
As for the lower bound,
Aichholzer et al.\ recently proved that,
for all point sets of size $n$ in general position, the number of triangulations is $\Omega(2.63^n)$
(the first result of this kind, $\Omega(2.33^n)$,
was proven in~\cite{ahn}).

Next we recall the definition of the double circle of size $n$,
which we denote by $\mathrm{DC}_n$.
For the sake of simplicity, we restrict ourselves to even $n$.
In this case, $\mathrm{DC}_n$ consists of $n/2$ points,
denoted by $P_1, P_2, \dots, P_{n/2 }$, in convex position;
and $n/2 $ points,
$Q_1, Q_2, \dots, Q_{n/2}$,
such that for each $i$, $1 \leq i \leq n/2 $,
$Q_i$ lies in the interior of the convex hull of $\{P_1, P_2, \dots, P_{ n/2 }\}$, very (``infinitesimally'') close to the midpoint of $P_i P_{i+1}$\footnote{
By convention, $P_{n/2 +1}=P_1$.}.
Figure~\ref{fig:dc}(a) shows $\mathrm{DC}_{12}$ and one of its
triangulations.

\begin{figure}
\begin{center}
\includegraphics[scale=0.8]{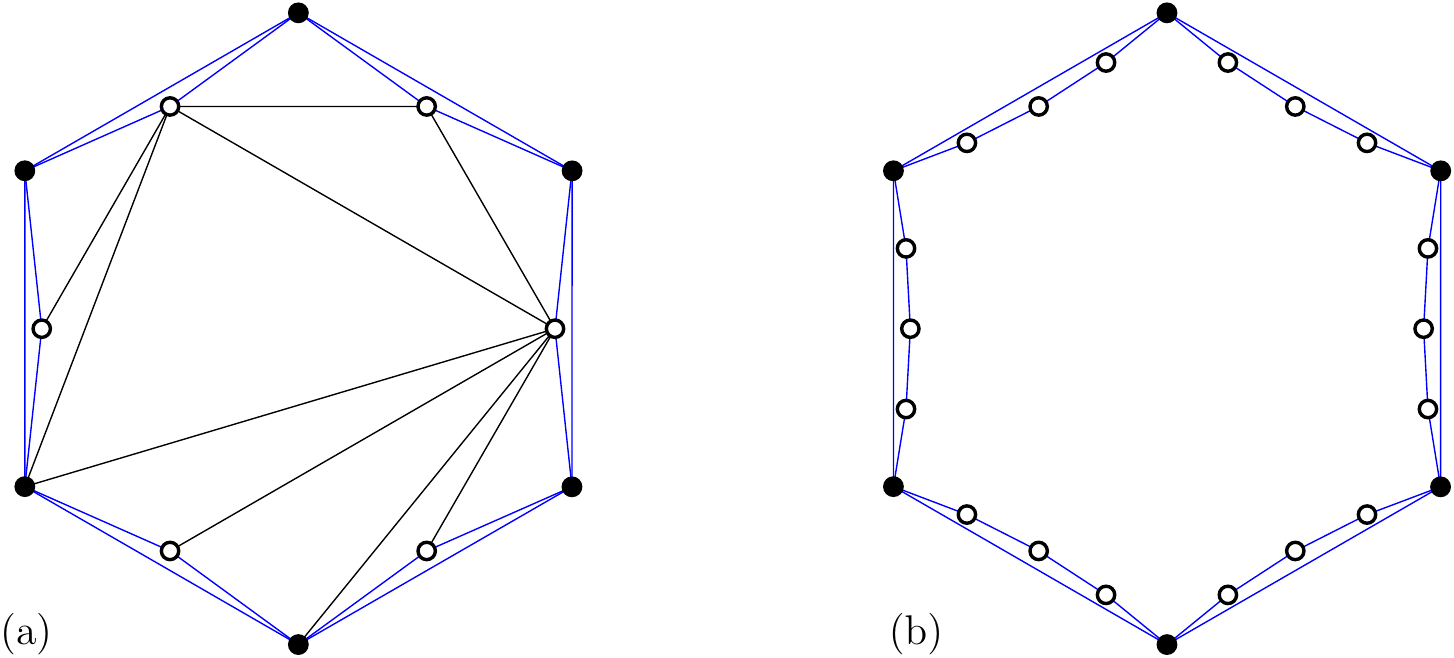}
\end{center}
\caption{(a) Double Circle of size $12$. (b) A generalized configuration.
Unavoidable edges are shown in blue colour.}
\label{fig:dc}
\end{figure}

Notice that each triangulation of $\mathrm{DC}_n$
necessarily uses the edges $Q_iP_i$ and $Q_iP_{i+1}$
for each $i$, $1 \leq i \leq n/2 $,
and, of course, all the edges that form the boundary of its convex hull.
Therefore we refer to them as \emph{unavoidable edges}.
In Figure~\ref{fig:dc}, unavoidable edges are shown in blue colour.
This observation leads to a simple bijection between
$\mathsf{TR}(\mathrm{DC}_n)$ and $\mathsf{TR}(\mathrm{SC}(n/2, 2))$:
given a triangulation of $\mathrm{DC}_n$,
move all the points $Q_i$ ``outwards'',
until they lie on the segments $P_i P_{i+1}$.
Thus, from this point of view, triangulations of $\mathrm{DC}_n$
are equivalent to triangulations of $\mathrm{SC}(n/2, 2)$, and the
above cited
bound $\mathsf{tr}(\mathrm{DC}_n)=\Theta^*(\sqrt{12}^{\,n})$
is a special case of our Theorem~\ref{thm:kinf} for
$r=2$, $k=n/2 \to \infty$.

Our goal was to investigate whether
the number of triangulations can decrease if
one inserts more points between the corners.
A similar idea, applied to the so-called \textit{double chain},
led to an improvement of the lower bound on the \emph{maximum} number of triangulations~\cite{d} and of perfect matchings~\cite{ar}.

Let us define our construction precisely.
For fixed $k$ and $r$,
we take $\mathrm{SC}(k, r)$ and slightly pull the inner points of the strings into the convex hull
so that, after this transformation, they lie
on circular arcs
of sufficiently big radius. This radius is chosen so that
the orientation of triples of points
which do not belong to the same string
is not changed.
See Figure~\ref{fig:dc}(b) for an illustration.
We refer to this construction as
\emph{indented $\mathrm{SC}(k, r)$} and denote it by $\mathrm{ISC}(k, r)$.
Notice that for $r=2$ we have the double circle:
$\mathrm{ISC}(k, 2) = \mathrm{DC}(2k)$.
Observe that the segments that connect consecutive points of a string
of $\mathrm{ISC}(k, r)$
are unavoidable for triangulations.
Together with the segments that form the boundary of the convex hull,
they split the convex hull into $k+1$ regions:
$k$ regions, each bounded by $r+1$ points in convex position,
and one region whose triangulations
are essentially equivalent to
triangulations of $\mathrm{SC}(k, r)$.
Due to this fact, the analysis of the number of triangulations
of $\mathrm{ISC}(k, r)$
is now easy: we have
$\mathsf{tr}(\mathrm{ISC}(k, r)) = \mathsf{tr}(\mathrm{SC}(k, r))\cdot C_{r-1}^{k}$.
By our asymptotic result in Theorem~\ref{thm:kinf},
we see that the exponential growth factor of the number of triangulations of
$\mathrm{SC}(k, r)$ as $k\to\infty$ --- and thus
the total number $n=kr$ of points tends to infinity ---
is $2(r+1)^{1/r}$.\footnote{
This result is also stated in~\cite{d};
however, the argument given there is non-rigorous since it relies
on~\cite[Theorem 3]{hn}
which holds for {\it fixed\/} $k$ rather than for $k \to \infty$.}
Hence
the growth factor for the number of triangulations of
$\mathrm{ISC}(k, r)$ equals
$2(r+1)^{1/r}C_{r-1}^{1/r}$.
This expression
is minimal for $r=2$, that is, for the
double circle.
If, on the other hand, we keep $k$ fixed
and let $r$ tend to infinity --- so that again
the total number $n=kr$ of points tends to infinity --- then similar
reasoning using our asymptotic result in Theorem~\ref{thm:rinf}
leads to the conclusion that the exponential growth factor
of the number of triangulations of
$\mathrm{ISC}(k, r)$ is~$8$.
Thus, somewhat disappointingly, the asymptotic count of
$\Theta^*(\sqrt{12}^{\, n})$ attained by $\mathrm{DC}(n)$
cannot be improved
by using balanced generalizations of the double circle,
in whatever way $n \to \infty$.

\smallskip

Let us return to the case of fixed $r$ and $k \to \infty$.
As stated above, the exponential growth factor in this case is
$g_r := 2(r+1)^{1/r}C_{r-1}^{1/r}$.
As $r \to \infty$, we have $(r+1)^{1/r} \searrow 1$
and $C_{r-1}^{1/r} \nearrow 4$, in both cases
monotonically for $r \geq 1$.
Thus, the fact $g_2 < g_1$ can be interpreted intuitively as follows:
when we pass
from $r=1$ to $r=2$,
the former expression decreases, while
the $k$ regions in convex position are just triangles with the unique (trivial) triangulation,
and so there is no extra factor.
On the other hand, for $r=3$ these $k$ regions are convex quadrilaterals with two triangulations,
and, as calculations above show,
their ``positive'' contribution to the total number of triangulations already dominates over
the ``negative'' contribution of the central region.
For $r \geq 3$, this tendency holds monotonically,
and, thus, $g_r$ has its minimum at $r=2$.

However, if one extends the expression $g_r$ for \textit{real} values of $r$
by using the Gamma function in the definition of Catalan numbers
(namely, $C_n = \frac{\Gamma(2n+1)}{\Gamma(n+1)\Gamma(n+2)} $),
one can observe that $g_r$ has its
minimum not at $r=2$ but rather at $r \approx 1.4957$.
This may lead to the idea that,
perhaps, 
we may get less triangulations if we
``mix" sides subdivided by one point (corresponding to $r=2$)
and non-subdivided sides (corresponding to $r=1$).
More precisely, let us consider
a subdivided convex polygon
in which $s$ sides are subdivided by one point,
all other sides are not subdivided,
and the total number of points is $N$
(where $N \geq 2s$).
We denote this partially subdivided polygon by $\mathrm{C}(N,s)$, and
its number of triangulations by $\trNk Ns$.
(Recall from the introduction that, by \cite{hn}, this number does not
depend on the specific distribution of the subdivisions among the sides
of the polygon.)

Proceeding in analogy with the inclusion-exclusion argument in
Section~\ref{sec:formula}, we observe that
the number of ways to choose $m$ pairwise non-crossing
essentially forbidden diagonals in $\mathrm{C}(N,s)$ is
$\binom{s}{m}$.
Once $m$ essentially forbidden diagonals of $\mathrm{C}(N,s)$
are chosen,
we are left with a convex $(N-m)$-gon to be triangulated.
Therefore, the number of illegal triangulations that use at least
$m$ essentially forbidden diagonals
is ${a}_{N,s,m} C_{N-m-2}$.
We apply the inclusion-exclusion principle to get
\begin{equation*}
\trNk{N}{s} = \sum_{m=0}^{ s} (-1)^m \, {a}_{N,s,m} \,
C_{N-m-2}= \sum_{m=0}^{ s} (-1)^m \, \binom sm \,
C_{N-m-2}.
\end{equation*}
Thus, the analogue of \eqref{eq:tr1} in the current context reads
\begin{equation} \label{eq:tr1A}
\trNk{N}{s}=
\frac {1} {2\pi i}\int_{\mathcal C}
\frac {dx} {2x^{N}}
(1-x)^s
\left({1-\sqrt{1-4x}}\right),
\end{equation}
where $\mathcal C$ is a small contour encircling the origin
once in positive direction. The substitution
$x= t(1-t)$, followed by the arguments used in the proof of
Proposition~\ref{prop:Formel}, turns this into
\begin{equation} \label{eq:intFormelA}
\trNk{N}{s}=
-\frac {1} {4\pi i}\int_{\mathcal C}
\frac {(1-2t)^2\,dt} {t^{N}(1-t)^{N}}
\left(
1-t+t^2
\right)^s.
\end{equation}

Deformation of the contour as described in the proof of
Proposition~\ref{prop:saddle} then leads us to the following integral
representation of $\trNk{N}{s}$.

\begin{proposition} \label{prop:saddleA}
For all positive integers $N$ and $s$ with $N\ge3$ and $N\ge 2s$, we have
\begin{equation} \label{eq:intFormel2A}
\trNk{N}{s}=
\frac {4^{N-s}\,3^s} {\pi }\int_{-\infty}^\infty
\frac {u^2\,du} {(1+4u^2)^N}
\left(
1-\tfrac {4} {3}u^2
\right)^s.
\end{equation}
\end{proposition}

Finally, following the proof of Theorem~\ref{thm:kinf}, we obtain
the following asymptotic estimate for $\trNk{N}{s}$, where both
$N$ and $s$ tend to infinity under the condition of approaching a fixed ratio.

\begin{theorem} \label{thm:kinfA}
Let $\alpha$ be a real number with $0\le \alpha\le 1/2$. Then we have
\begin{equation} \label{eq:kinfA}
\trNk{N}{s} =
\frac {\big(4^{1-\alpha}3^\alpha\big)^{N}}
{16\sqrt\pi (1+\frac {\alpha} {3})^{3/2}N^{3/2}}
\big(1+o(1)\big),
\quad \quad \text{as }N,s\to\infty\text{ subject to } s/N\to\alpha.
\end{equation}
\end{theorem}

As is obvious from this asymptotic formula, the minimal exponential
growth is attained for the maximal possible $\alpha$, that is, for
$\alpha=1/2$. The corresponding polygon is again the double circle.

In summary, our results provide further
support for the conjecture
of Aichholzer, Hurtado and Noy
that, asymptotically,
the double circle yields the minimal number of
triangulations of $n$ points in general position.

\end{document}